\documentclass[11pt,reqno]{amsart}
\usepackage{amsfonts,amsmath,amscd,amssymb}
\usepackage[T2A]{fontenc}
\newtheorem{thm}{Theorem}[section]
\newtheorem{proposition}[thm]{Proposition}
\newtheorem{corollary}[thm]{Corollary}
\newtheorem{lemma}[thm]{Lemma}
\newtheorem{definition}[thm]{Definition}
\newtheorem{example}[thm]{Example}
\newtheorem{remark}[thm]{Remark}
\vspace{4ex}

\title{A new approach to the representation theory of the symmetric groups.
IV. $ \Bbb Z_{2}$-graded groups and algebras}

\author{A.~M.~Vershik \and A.~N.~Sergeev}
\address{St.~Petersburg Department of Steklov Institute of Mathematics}

\begin{document}

\maketitle

\begin{abstract}
We start with definitions of the general notions of the theory of $\Bbb Z_{2}$-graded algebras.
Then we consider theory of inductive families of $\Bbb Z_{2}$-graded semisimple finite-dimensional
algebras and its representations in the spirit of approach of the papers \cite{VO,OV} to
representation theory of symmetric groups. The main example is the classical - theory of the
projective representations of symmetric groups.
\end{abstract}

\tableofcontents

\section{Introduction}
In this paper we formulate the main notions of the theory of locally
semisimple $\Bbb Z_{2}$-graded finite-dimensional algebras and their representations.
The main result is a translation of the inductive method
of constructing the representation theory (the method
of Gelfand--Tsetlin algebras) developed in
\cite{OV, VO, V} to the case of $\Bbb Z_{2}$-graded algebras. In particular,
it allows us to use this method for constructing projective
representations of the symmetric groups.

Although the general theory of $\Bbb Z_{2}$-graded semisimple algebras and
their representations was partially described in
\cite{KL, J}, nevertheless we present a systematic treatment of the
main notions of this theory, keeping in mind their application to inductive chains
of semisimple $\Bbb Z_{2}$-graded algebras. In this case, we have a number of interesting
new phenomena, which are related to the fact that
there are two types of simple
$\Bbb Z_{2}$-graded finite-dimensional algebras (algebras of type
$M(n,m)$ and $Q(n)$) and, consequently, two types of
simple modules, and the branching graph of an inductive family
has an additional structure (involution). The central question
concerns the relation between
representations of a graded algebra (group) and representations of its even
part, for example, representations of symmetric and alternating groups.
The answer to this question is comparatively simple: there is a bijection
between these two classes of modules
(see Proposition~\ref{eqv}), and this relation can be used in both directions.
The classical example is the description of representations of
the alternating groups via representations of the symmetric groups.
The converse problem, which is much more difficult,
is to describe projective representations of the symmetric groups. We
investigate this problem below.
Surprisingly enough, the knowledge of ordinary representations of the symmetric group
does not help to solve this problem
and in fact is not used in the related papers.
In other interesting examples, both representations of a
$\Bbb Z_{2}$-graded algebra and representations of its even part are unknown, and
we look for them simultaneously. Moreover, sometimes, having
an explicit description of a graded algebra, we cannot explicitly describe its even part.

We suggest to construct representations of inductive chains
of $\Bbb Z_{2}$-graded groups and algebras using the same techniques
as were used in \cite{OV,VO,V} for describing representations of the symmetric groups;
this is the main goal of the paper. One example that we consider in more
detail is the problem of describing projective representations of the symmetric group.
This problem is reduced to that of describing simple modules of a certain
$\Bbb Z_{2}$-graded algebra which was considered by I.~Schur.

Projective representations of the symmetric groups were studied by
many authors (e.g., \cite{Mo,Se3,N1,N2}). We use this example to illustrate
the new approach to the problem concerning representations of
$\Bbb Z_{2}$-graded chains of algebras. First of all, we find conditions
under which the branching of representations of a chain of semisimple
$\Bbb Z_{2}$-graded algebras is simple. In what follows, the most important part
is played by a generalization of the notion of Gelfand--Tsetlin algebra
for a chain
$\frak A=<A(1)\subseteq A(2)\subseteq \dots \subseteq A(n)>$
of semisimple $\Bbb Z_{2}$-graded algebras. First of all, one should generalize
the notion of center and define the so-called supercenter of a
$\Bbb Z_{2}$-graded algebra. For a chain of $\Bbb Z_{2}$-graded algebras,
the notion of the Gelfand--Tsetlin algebra splits into several notions,
because, in contrast to the nongraded case, the algebra
$SGZ(\frak Y)$ generated by the supercentralizers of the successive subalgebras
(which in what follows will be called the {\it Gelfand--Tsetlin
superalgebra}) does not coincide with the algebra
$SZ(\frak Y)$ generated by the supercenters of the algebras
$A(k)$. In general, the algebras $SGZ(\frak Y)$ are not commutative,
even in the case of a simple branching, but their structure turns out
to be standard: it is the tensor product of a commutative algebra and a
Clifford algebra. Between the algebra
$SGZ(\frak Y)$ and its supercenter
$SZ(\frak Y)$ there is a place for the ordinary Gelfand--Tsetlin algebra
$GZ(\frak Y)$ of the even part of the chain
$\frak Y$. The analysis of representations and characters of these algebras
is the essence of the method.
Similarly to
\cite{OV,VO}, we find the spectrum (the list of irreducible representations)
of the algebra $SGZ(\frak Y)$. This is done by using essentially
the same technique: we
reduce the problem to the description of an analog of the Hecke algebra, which in turn
allows us to describe the irreducible representations of the algebra
$A(n)$. This method is used for describing the projective representations
of the symmetric groups, constructing an analog of Young's forms, bases, etc.
As in \cite{OV,VO}, the so-called strict Young tableaux, which parameterize
a distinguished basis
in representations,
turn out to be the points of the spectrum
of an appropriate Gelfand--Tsetlin algebra, and the irreducible projective
representations of $S_n$ are indexed by the strict diagrams, i.e., the orbits
of admissible substitutions of points of the spectrum. This description
of projective representations is one of the goals of the paper.

In the representation theory of
$ \Bbb Z_{2}$-graded groups and algebras and their branching diagrams,
there arise many combinatorial problems, which apparently have not
yet been studied. Even the quite well-known question about representations
of the alternating group $S_n^+$ regarded as the even part of the
 $\Bbb Z_{2}$-graded group $S_n$ is not sufficiently studied from this point of view.
Note that, for example, a description of the alternating group similar
to the description of Coxeter groups (more exactly, its presentation
as a ``local group'' in the sense of \cite{V1}) was obtained only
quite recently in \cite{VV}. This presentation
should be used for obtaining a direct construction of representations
of $S_n^+$ independently of representations of  $S_n$.
To this end,
one should develop and apply the whole machinery of Gelfand--Tsetlin
algebras. We hope to return to this problem in another paper. By analogy
with the symmetric groups, one should consider projective representations
of other classical Weyl groups, construct a normal form, describe the subalgebra
generated by the supercentralizers, etc. One may also hope that the ideology
of an inductive construction of representation theory will be applicable
also in the theory of superalgebras, in particular, Lie superalgebras.

Let us briefly describe the contents of the paper. The second section
contains the main definitions of the theory of
$\Bbb Z_{2}$-graded associative finite-dimensional semisimple algebras.
In particular, we describe the structure of these algebras and give a description
of simple algebras. Like in the classical case, an important role is played
by the notions of center and centralizer, which in the graded case have
several versions. The third section contains a brief treatment of the theory
of modules over $\Bbb Z_{2}$-graded algebras. Here the main theorem reduces
the description of graded modules over a graded algebra to the description
of nongraded modules over some other algebra. We also present theorems
describing the relations between graded modules and nongraded modules;
graded modules over an algebra and nongraded modules over its
even part. Note that a brief and clear introduction to the
theory of associative superalgebras and modules over them
can be found in \cite{KL}. In the fourth section we introduce one of the main
objects of the paper: inductive families of
$\Bbb Z_{2}$-graded semisimple finite-dimensional algebras. We define
the branching graph of such a family
and prove a theorem characterizing these graphs. Besides,
we show that the branching graphs of simple modules and
the branching graphs of inductive families
of algebras, which {\it a priori} are defined in different ways, coincide.
We also give a criterion for the simplicity of branching, which
is based on the notion of graded centralizer.

The fifth section is devoted to Gelfand--Tsetlin algebras. Here, in contrast to the
nongraded case, there arise several natural analogs of this algebra. The
main result of this section is a theorem describing the relation
between representations of the Gelfand--Tsetlin superalgebra and representations
of the original algebra. In the case of a simple branching, these results
are interpreted in terms of the corresponding branching graph. In particular,
this allows us to define a natural equivalence relation on the path
space of a graded graph.

In the most important last section, we present an application of the theory
developed in the previous sections to the study of projective representations of the
symmetric groups. We explicitly describe the Gelfand--Tsetlin superalgebra
in terms of odd analogs of the Young--Jucys--Murphy elements, which allows us
to describe the spectrum of this superalgebra
following the method of
\cite{OV}.

\section{Main definitions}

In what follows, we assume that the ground field is  $\mathbb C$.

Recall the definition of a  $\Bbb Z_{2}$-graded algebra.

\begin{definition}
A $\Bbb Z_{2}$-graded algebra is an algebra $A$ that has a direct sum decomposition
$A=A_{0}\oplus A_{1}$ such that if $a\in
A_{i}$, $b\in A_{j}$, then $ab\in A_{i+j}$, for $i,j\in\Bbb Z_{2}$.
We will also write the condition $a\in A_{i}$ in the form
$p(a)=i$ and call $p$ the parity function. The even part
$A_0$ is a subalgebra of $A$, and the odd part $A_1$ is an $A_0$-module.
\end{definition}

An equivalent definition is as follows:  a $\Bbb Z_{2}$-graded algebra is an algebra
$A$ with an automorphism
 $\theta$ such that $\theta^2=1$. Here $ A_{0}$ is the eigensubspace corresponding
to the eigenvalue $1$, and
$ A_{1}$ is the eigensubspace corresponding
to the eigenvalue $-1$. We will refer to
$\theta$ as the parity automorphism.

A homomorphism of $\Bbb Z_{2}$-graded algebras is a homomorphism
of ordinary algebras that is also grading-preserving.

{\it An algebra  $A$ regarded without the grading will be denoted by
$|A|$.}

A subalgebra of a $\Bbb Z_{2}$-graded algebra $A$ is a subalgebra
 $B \subset A$ in the ordinary sense that inherits the grading:
  $B=B_{0}+B_{1}$,
$B_0=B\cap A_{0}$, $B_1=B\cap A_{1}$.

In a similar way, a two-sided ideal in a $\Bbb Z_{2}$-graded algebra $A$
is a two-sided ideal $I$ that inherits the grading:
$I=I\cap A_{0}+I\cap A_{1}$. Analogously, we define the notions of left and
right ideals.

As we will see below, the following algebra is useful when considering
  $\Bbb Z_{2}$-graded algebras:
 $$
 A[\varepsilon]=\{a+b\varepsilon\mid a, b\in A, \; \varepsilon^2=1,
 \; \varepsilon a=\theta(a)\varepsilon\}.
 $$

\begin{definition} A $\Bbb Z_{2}$-graded algebra is called
{\it simple} if it contains no two-sided
$\Bbb Z_{2}$-graded ideals other than $0$ and the algebra itself.
\end{definition}

 \begin {definition} A $\Bbb Z_{2}$-graded algebra $A$ is called
  {\it semisimple} if every two-sided graded ideal of $A$ has a
  graded complement, i.e., for every such ideal $I$ there exists a two-sided graded ideal
$I^{\prime}$ such that $ A=I\oplus I^{\prime}$.
 \end {definition}

  \begin{example}\label{M}\footnote{Regarding the algebras below as
  Lie superalgebras, we obtain two principal series of Lie superalgebras:
    $\mathfrak{gl}(n,m)$ and $\mathfrak{q}(n)$.} {\rm For $n,m\ge0$, denote by $M(n,m)$
the set of matrices of the following form:
$$
M(n,m)_{0}=\left\{\begin{pmatrix}A&0\\
0&D\end{pmatrix}\right\},\quad M(n,m)_{1}=\left\{\begin{pmatrix}0&B\\
C&0\end{pmatrix} \right\},
$$
where $A$ is a square matrix of order $n$, $D$ is a square matrix of order
$m$, and $B,C$ are rectangular matrices of orders
$n\times m$ and $m\times n$, respectively. Here the parity automorphism $\theta$
is an inner automorphism: there is
a unique (up to sign) element $J\in M(n,m)_{0}$ satisfying
$J^{2}=1$ such that $\theta(M)=JMJ^{-1}$; namely,
$$
J=\begin{pmatrix}1_{n}&0\\
0&-1_{m}\end{pmatrix},
$$
where $1_k$ is the identity matrix of order $n$.
}\end{example}

It is not difficult to check that the algebra
$M(n,m)$ is simple both as a $\Bbb Z_{2}$-graded algebra and as
a nongraded algebra.

\begin{example} {\rm Denote by $Q(n)$ the set of matrices of the following form:
$$
Q(n)_{\bar0}=\left\{\begin{pmatrix}A&0\\
0&A\end{pmatrix}\right\},\quad Q(n)_{\bar1}=\left\{\begin{pmatrix}0&B\\
B&0\end{pmatrix} \right\},
$$
where $A,B$ are square matrices of order  $n$ and the parity automorphism
is of the same form as in the previous example; note that
$J\notin Q$.}
\end{example}

It is not difficult to check that
$Q(n)$ is simple as a graded algebra and is not simple as a nongraded algebra.

The following well-known theorem describes the structure of simple
$\Bbb Z_{2}$-graded algebras; its proof can be found, e.g., in
\cite{J,KL}.

\begin{thm}\label{structure}Every finite-dimensional simple
$\Bbb Z_{2}$-graded algebra over
$\mathbb C$  is graded isomorphic either to
$M(n,m)$ or to  $Q(n)$, where $n,m$ are arbitrary positive integers.
\end{thm}

The previous theorem shows that there are more
$\Bbb Z_{2}$-graded simple finite-dimensional algebras
than nongraded ones. Nevertheless, it turns out that the class of semisimple
finite-dimensional algebras does not depend on the grading
\cite{J,KL}.

\begin{thm}\label{semisimple}

$1)$ A finite-dimensional graded algebra is semisimple as a graded algebra
if and only if it is semisimple as a nongraded algebra.

$2)$ Every finite-dimensional semisimple graded algebra is the sum
of finitely many simple algebras, i.e.,
$M(n,m)$ and $Q(k)$.
\end{thm}

\begin{corollary} The even part of a graded semisimple finite-dimensional
algebra is semisimple.
\end{corollary}

\begin{corollary} The number of simple components of a semisimple
finite-dimensional graded algebra is equal to the dimension of the even
part of its ordinary center, and the number of simple components of type $Q$
is equal to the dimension of the odd
part of its ordinary center.
\end{corollary}

\begin{remark}\label{chet}{\rm It follows from the previous theorem that every
finite-di\-men\-sion\-al semisimple graded algebra $A$ admits a unique
decomposition of the form
$A=A_{M}\oplus A_{Q}$, where
$A_{M}$  is the sum of all simple components of type $M$ and
$A_{Q}$ is the sum of all simple components of type $Q$.
Further, as follows from Example~\ref{M}, there exists an element $J_{A}\in A_{M}$ such that
$\theta(a)=J_{A}aJ_{A}^{-1}$, $a\in A_{M}$.}
\end{remark}

One of the most important notions of the classical theory of algebras
is that of the center of an algebra. In the graded case, a natural analog
of the center is the following algebra, which arises from considering
the center of the algebra $A[\varepsilon]$.

\begin{definition} The center of a graded algebra $A$ is the algebra that is the sum
of the even part
of the center of the nongraded algebra and the even part of the twisted center:
$$
Z(A)=Z(|A|)_{0}+Z^{\theta}(|A|)_{0},
$$
where $Z(|A|)_{0}=\{a\in A_{0}\mid ab=ba\text{ for every } b\in A\}$
is the even part of the ordinary center
and $Z^{\theta}(|A|)_{0}=\{a\in A_{0}\mid ab=\theta(b)a\text{ for every } b\in A\}$
is the even part of the twisted center.
\end{definition}

\begin{definition} The graded centralizer of a subalgebra $B$ in a graded  algebra
$A$ is the algebra that is the sum
of the even part of the nongraded centralizer and the
even part of the twisted centralizer:
$$
Z(A,B)=Z_{0}(|A|,|B|)+Z^{\theta}_{0}(|A|,|B|),
$$
where  $Z_{0}(|A|,|B|)=\{a\in A_{0}\mid ab=ba\text{ for every } b\in B\}$
is the even part of the ordinary centralizer and $Z_{0}^{\theta}(|A|,|B|)
=\{a\in A_{0}\mid ab=\theta(b)a\text{ for every } b\in B\}$
is the even part of the twisted centralizer.
\end{definition}

Note that if the algebra $A$ is finite-dimensional and semisimple, then
its graded center coincides with the center of its even part:
$Z(A)=Z(A_{0})$. This can easily be checked by reducing to the case
of a simple algebra. A similar assertion is also true for the centralizer,
see Lemma~\ref{centraliser}.

The following notion, which is used in the theory of Lie superalgebras, turns
out to be useful.

 \begin{definition}\label{super}
The supercentralizer of a subalgebra
$B\subset A$ is the algebra
$$
SZ(A,B)=
\{a\in A\mid ab=(-1)^{p(a)p(b)}ba\text{ for every }b\in B\}.
$$
In particular, the supercentralizer  $SZ(A,A)$ is called the supercenter of
$A$ and is denoted by $SZ(A)$.
\end{definition}

Note that for a semisimple finite-dimensional algebra, the supercenter
coincides with the even part of the ordinary center:
$SZ(A)= Z(|A|)_{0}$.

We will also need the notion of a
$\Bbb Z_{2}$-graded group.

\begin{definition}
A finite group $G$ is called $\Bbb Z_{2}$-graded if it contains
a distinguished normal subgroup
$G_{0}$ of index $2$, and the decomposition into the cosets of
$G_{0}$ is the decomposition into the even and odd parts:
$G=G_{0} \bigcup G_{1}$.
\end{definition}

The group algebra of a  $\Bbb Z_{2}$-graded group has a natural
$\Bbb Z_{2}$-grading. It is not difficult to prove the following fact.

\begin{thm} The group algebra of a finite graded group with the natural
$\Bbb Z_{2}$-grading is a semisimple
$\Bbb Z_{2}$-graded algebra.
\end{thm}

This theorem is a special case of a more general result proved in
\cite{NS}.

For brevity, in what follows we often use the term ``a graded algebra (group,
module)'' instead of
``a $\Bbb Z_{2}$-graded algebra (group,
module).'' Here are the main examples of graded groups and algebras.

\medskip
 $1)$ The symmetric group $S_n$ in which the parity of an element
 $g \in S_n$ is defined as the parity of the permutation $g$, and
 the group algebra ${\Bbb C}[S_n]$ endowed with the corresponding grading.

\medskip
 $2)$ The $\Bbb C$-algebra  $\frak A_{n}$ corresponding to projective
 representations of the symmetric group $S_n$; it is generated by elements
 $\tau_{1},\dots,\tau_{n}$ satisfying the relations
$$
\tau_{k}^2=1,\quad(\tau_{k}\tau_{k+1})^3=1,\quad(\tau_{k}\tau_{l})^2=-1
\;\text{ if }\; \mid k-l\mid >1.
$$
The graded algebra $\frak A_{n}$ is not the group algebra of any graded group.

\medskip
 $3)$ The semidirect product $S_{n}\ltimes C_{n}$ of the Clifford algebra
 $C_n$ (with the natural action of the symmetric group) and the symmetric group
 $S_n$;
 here all generators
 of the Clifford algebra are assumed to be odd, and all elements of the symmetric
 group are assumed to be even. Note that we have the relation
 $$
  S_{n}\ltimes C_{n}=\frak A_{n}\otimes C_{n},
 $$
where the tensor product in the right-hand side is understood as the
tensor product of
graded algebras. The corresponding definition is as follows.

\smallskip
\begin{definition}
The tensor product $A\otimes B$ of graded algebras
$A$ and $B$ is the graded algebra that coincides as a vector space
with the ordinary tensor product of the algebras $A$ and $B$, with
the parity automorphism and multiplication defined by the following rules:
$$
\theta_{A\otimes B}(a\otimes b)=\theta_{A}(a)\otimes\theta_{B}(b),\quad
(a\otimes b)(a^{\prime}\otimes b^{\prime})=(-1)^{ji}aa^{\prime}\otimes bb^{\prime},
$$
where $a^{\prime}\in A_{i}$, $b\in B_{j}$.
\end{definition}

Note that the tensor product of graded algebras, regarded
as a nongraded algebra, is not, in general, isomorphic to their tensor product
as ordinary algebras.

\begin{lemma} The tensor product of semisimple finite-dimensional associative graded
algebras is semisimple.
\begin{proof}
The proof reduces to the case of a simple graded algebra. In this case,
it is easy to check the equalities
$$
M(n,m)\otimes M(n^{\prime},m^{\prime})=M(nn^{\prime}+mm^{\prime},nm^{\prime}+nm^{\prime}),
$$
$$
M(n,m)\otimes Q(l)=Q((m+n)l),\quad Q(n)\otimes Q(m)=M(nm,nm).
$$
\end{proof}
\end{lemma}

\section{$\Bbb Z_{2}$-graded modules}

\def\GMod{{\rm GMod}}
\def\Gdim{\operatorname{Gdim}}

We will consider the category $\GMod$ of graded modules
over a $\Bbb Z_{2}$-graded algebra $A$.

\begin{definition}
A graded module over a graded algebra $A$ is an $A$-module
$V$ that has a decomposition
$V=V_{0}\oplus V_{1}$ such that the corresponding grading is preserved by
the action of $A$, i.e.,
$A_{i}V_{j}\subset V_{i+j}$, where $i,j\in\Bbb Z_{2}$.
\end{definition}

Recall also that the graded dimension
$\Gdim V$ of a graded module $V$
is the pair $(\dim V_{0}, \dim V_{1})$.
The set of morphisms in the category
$\GMod$ is the set of homomorphisms $ f:V\rightarrow U$
of ordinary modules that preserve the grading, i.e.,
satisfy $f(V_{i})\subset U_{i}$, $i\in \Bbb Z_{2}$.

On the category $\GMod$ there exists a natural functor of changing the parity:
$V\rightarrow P(V)$, where $P(V)_{0}=V_{\bar1}$, $P(V)_{1}=V_{\bar0}$,
and the new action coincides with the old one.

\begin{definition}
A graded module $V$ is called simple or irreducible if it has no
graded submodules other than the zero one and $V$ itself.
\end{definition}

\smallskip
Examples.

\medskip
 1) The algebra $M(n,m)$ has
 two nonisomorphic
 graded simple modules
 $V$ and $ P(V)$. Note that the graded dimension of one of them equals
 $(n,m)$, while the graded dimension of the other one equals $(m,n)$. It is easy to check that both modules
 are simple and isomorphic as nongraded modules. Such modules will be called
 a pair of {\it antipodal  modules} of type $M$.
 Note that the algebra $M(n,m)$ is the algebra of all graded endomorphisms
 of the $\Bbb Z_2$-graded vector spaces $V$ and $P(V)$.

\medskip
 2) The algebra $Q(n)$ has only one simple graded module
  $ V=P(V)$, where $V$ is the standard representation and
$\Gdim V=(n,n)$. It is easy to check that $V$ is reducible
as a nongraded module and that it is the direct sum of two nonisomorphic simple
nongraded modules. Irreducible modules of this type will be called
modules of type $Q$.

\smallskip
\begin{lemma}\label{centraliser} If $A $ is a semisimple finite-dimensional
algebra and $B$ is a semisimple subalgebra of $A$, then
$$
Z(A,B)=Z(A_{0},B_{0}).
$$
\begin{proof}
Clearly, $Z(A,B)=Z_{0}(|A|,|B|)+ Z_{0}^{\theta}(|A|,|B|)\subset Z(A_{0},B_{0})$.
Thus it suffices to show that both algebras have the same
dimension over $\mathbb C$. Further, it is not difficult to check that
we may restrict ourselves to the case where $A$ is a simple graded algebra,
i.e., $A=M(n,m)$ or $A=Q(n)$. Consider the first case. Let
$V$ be one of the irreducible modules over $A$. We may restrict ourselves
to the case where $V$ is the sum of $B$-modules
$U,P(U)$ for some irreducible $B$-module $U$. The following
cases are possible: 1) $B=B_{0}$;
 2) $B\ne B_{0}$.

In the first case, $Z_{0}(A,B)=Z_{0}^{\theta}(|A|,|B|)=Z(A_{0},B_{0})$.
In the second case, we have two possibilities:  $U\ne P(U)$ and  $U=P(U)$.

In the first one, $V=U^k\oplus P(U)^l$ for some $k,l$, whence
$\dim(Z_{0}(|A|,|B|)+Z_{0}^{\theta}(|A|,|B|))=2k^2+2l^2=\dim Z(A_{0},B_{0})$.
In the second one,  $V=U^k$ for some $k$, whence
$\dim(Z_{0}(|A|,|B|)+Z_{0}^{\theta}(|A|,|B|))=4k^2=\dim Z(A_{0},B_{0})$.
 \end{proof}
\end{lemma}

Let $\theta$ be an automorphism of an algebra $A$.
Then on the category of nongraded modules we have the natural
functor $V\rightarrow V^{\theta}$, where $V^{\theta}=V$
and the new action is defined by the formula
$a\star v=\theta(a)v$.

\def\Rep{{\rm Rep}}

\begin{corollary}
Let $A$ be a graded algebra, $\theta$ be the parity automorphism of $A$, and
$$
\Rep(|A|)=(E_{1},E^{\theta}_{1},\dots,E_{r},E^{\theta}_{r}, F_{1}=F^{\theta}_{1},\dots,F_{s}=F^{\theta}_{s})
$$
be the complete set of pairwise nonisomorphic nongraded modules. Then
$$
\Rep(A)=(F_{1},P(F_{1}),\dots,F_{s},P(F_{s}),E_{1}\oplus E^{\theta}_{1},\dots,E_{r}\oplus E^{\theta}_{r} )
$$
is the complete set of pairwise nonisomorphic graded modules.
\end{corollary}

It turns out that the category of graded modules is isomorphic
to a certain category of nongraded modules over some other algebra. Consider the algebra
 $$
 A[\varepsilon]=\{a+b\varepsilon\mid a, b\in A, \; \varepsilon^2=1,
 \; \varepsilon a=\theta(a)\varepsilon\}.
 $$
Note that  $A[\varepsilon]$
has the canonical automorphism
$\varphi$ defined by the formulas $\varphi(a)=a$, $a\in A$, and $\varphi(\varepsilon)=-\varepsilon$.

\begin{proposition} The category of nongraded  $A[\varepsilon]$-modules
is isomorphic to the category $\GMod$ of graded $A$-modules. Under
this isomorphism,
the functor of changing the parity goes to the functor
$V\rightarrow V^{\varphi}$. A graded irreducible
$A$-module is of type $M$ if and only if $V\ne V^{\varphi}$; it is of type
$Q$ if and only if $V=V^{\varphi}$.
\end{proposition}

\def\Mod{{\rm Mod}}
\def\End{{\rm End}}

 \begin{proof}
Let us construct functors
$$
F:\Mod( A[\varepsilon])\longrightarrow \GMod( A)\quad\text{and}\quad G:\GMod(A) \longrightarrow \Mod( A[\varepsilon]).
$$
Let $V$ be an object of $\Mod( A[\varepsilon])$; then
$\varepsilon\in \End(V)$. Since
$\varepsilon^2=1$, it follows that $V$ can be regarded as an object of the category
$\GMod(A)$
with the following grading:
$$
V_{0}=\{v\in V\mid \varepsilon v=v\}\quad\text{and}\quad
V_{1}=\{v\in V\mid \varepsilon v=-v\}.
$$
The functor $F$ acts identically on morphisms. Since every morphism commutes
with $\varepsilon$, it preserves the above grading and hence is
a morphism of $\GMod(A)$.

The inverse functor also acts identically on objects and morphisms; the action
of $\varepsilon$ coincides with the action of the parity operator. It
is not difficult to check that
$FG$ and $GF$ are identical functors on the corresponding categories.
\end{proof}

\begin{lemma}\label{epsilon}
Assume that a graded algebra contains an odd element
$p $ such that $p^{2}=1 $. Then one can introduce a grading
on the algebra $ A_{0} $ such that $A$ will be isomorphic to
$A_{0}[\varepsilon] $.
\end{lemma}
\begin{proof}
Take the automorphism $a\rightarrow pap^{-1}$ as the parity automorphism
on $ A_{0} $. Then  $A_{0}[\varepsilon]=A $.
\end{proof}

Let us say that a $\Bbb Z_{2} $-grading of a semisimple algebra is {\it
essential}
if for all its simple components of type $M(n,m)$ both indices
are positive: $n,m>0$. For such an algebra,
the odd and even parts of every simple module are nontrivial.

 \begin{proposition}\label{eqv}
Let $A$ be a semisimple finite-dimensional
$\Bbb Z_{2}$-graded algebra with an essential grading.
Then the functor
$$
I:V_{0}\longrightarrow A\otimes_{A_{0}} V_{0}
$$
from the category $\Mod(A_{0})$ to the category $\GMod(A)$
is an isomorphism of categories.
\end{proposition}

\begin{proof}
It suffices to prove that the restriction functor
$$
R:V\longrightarrow V_{0}
$$
from the category  $\GMod(A)$ to the category  $\Mod(A_{0})$
is the two-sided inverse to $I$. Indeed,
$R\circ I(V_{0})=A_{0}\otimes_{A_{0}}V_{0}=V_{0}$. Conversely, let us prove
the equality
$I\circ R(V)=V$. By additivity, it suffices to prove it
for a simple module $V$. But the previous equality is equivalent to
$A\otimes_{A_{0}}V_{0}=V$. Since $V$ is a simple
$A$-module, this equality implies that
$V_{0}$ is a simple
$A_{0}$-module. By the assumptions of the lemma, the even
part of every simple
$A$-module is nontrivial. But the even part of the module
$A\otimes_{A_{0}}V_{0}$ equals
$V_{0}$. Thus this module is irreducible and coincides with $V$.
The proposition is proved.
\end{proof}

It follows from the proposition that the theory of graded modules
of every semisimple $\Bbb Z_{2}$-graded algebra $A$ reduces
to the theory of ordinary representations of its even part $A_0$.
However, this reduction does not always help to describe all
representations of the algebra $A$, and, conversely, can be used
in the opposite direction. This is the case for projective
representations of the symmetric groups, where the even part
has no simple realization, see Section~6 below.

\section{Inductive families (chains) of $\Bbb Z_{2}$-graded algebras}

\begin{definition} A $\Bbb Z_{2}$-graded graph is a graph that has
an involutive automorphism.
\end{definition}

The main examples of $\Bbb Z_{2}$-graded graphs appear in the following situation.
Consider a (finite) chain of
$\Bbb Z_{2}$-graded algebras
 $$
\frak Y=< \mathbb C= A(1)\subset A(2)\subset\dots \subset A(n)>,\quad n=1,2, \dots.
 $$
Denote by $GA_{i}^{\wedge}$ the set of isomorphism classes
of irreducible objects of the category
$\GMod$.

\def\GHom{{\rm GHom}}

\begin{definition} The branching
graph
of simple modules of the chain $\frak Y$
is the directed graded\footnote{One should not
confuse the term ``graded'' in the context related to the branching graph
with the $\Bbb Z_{2}$-grading of algebras, modules, etc.}
graph $Y(\frak Y)$ whose set of vertices is the disconnected union
$$
\bigcup_{i=1}^n GA(i)^{\wedge}
$$
and the number of edges connecting a vertex
$U\in GA_{i}^{\wedge}$ with a vertex
$V\in GA_{i+1}^{\wedge}$ and directed from $U$ to $V$ is equal
to the dimension of the vector space $\GHom_{A_{i}}(U,V)$
of $A_i$-homomorphisms between $U$ and $V$.
\end{definition}

The corresponding branching graph of nongraded modules will be denoted by
 $Y(|\frak Y|)$. The branching graph of the even subalgebras will be denoted by
 $Y(\frak Y_{0})$.

On the category of $\Bbb Z_{2}$-graded modules we have the involutive
functor $V\rightarrow P(V)$ of changing the parity. Hence every graph of the form
$Y(\frak Y)$ has an automorphism
$\omega$ such that $\omega^2=1$. Thus $\omega$ determines the structure
of a $\Bbb Z_{2}$-graded graph on $Y(\frak Y)$. Besides, the zero level of every such graph
contains two vertices, which are swapped by
$\omega$. It turns out that these properties are characteristic.

A characterization of the graded branching graphs of ordinary modules
is as follows (see \cite{VK}).

\begin{thm}
Let $Y=\cup_{i=0}^n Y_{i}$ be a finite graded  graph with positive integer
multiplicities of edges. Then it is the branching graph of a chain of
nongraded modules if and only if the following conditions hold:

{\rm1)} For every vertex $y$ of $Y$ that does not belong to $Y_{n}$
there exists a vertex that immediately follows $y$.

{\rm2)} For every vertex $y$ of $Y$ that does not belong to $Y_{0}$
there exists a vertex that immediately precedes $y$.

{\rm3)} The set $Y_{0}$ consists of a single element.
\end{thm}

In the case of a chain of graded algebras, we have the following theorem.

\begin{thm}
Let $Y=\cup_{i=0}^n Y_{i}$ be a finite graded graph with positive integer
multiplicities of edges. Then it is the branching graph of a chain
of graded modules if and only if the following conditions hold.

{\rm1)} For every vertex $y$ of $Y$ that does not belong to $Y_{n}$
there exists a vertex that immediately follows $y$.

{\rm2)} For every vertex $y$ of $Y$ that does not belong to $Y_{0}$
there exists a vertex that immediately precedes $y$.

{\rm3)} There exists an involutive automorphism
$\omega$ of the graph $Y$ that preserves the grading.

{\rm4)} The set $Y_{0}$ consists of two elements, and the automorphism
$\omega$ swaps these elements.
\end{thm}

The construction of an algebra from the corresponding graph follows the same
scheme as in \cite{VK}. Namely, consider the set $B(Y)$ of loops of the graph $Y$,
i.e., the set of pairs of paths $(t,s)$ that begin at the same
vertex of level $0$ and end at the same vertex of level
$n$. Consider the vector space $A$ of functions $f$ on
$B(Y)$ that are invariant under $ \omega$, i.e., satisfy $f(\omega
s, \omega t)=f(s,t)$, and define the parity operator in this space by
the following formula:
$ \theta(f)(t, s) = f(t,s)$ if the paths
$s,t$ begin at the same vertex, and $\theta(f)(t, s) =-f(t,s)$
if the paths $s,t$ begin at different vertices. The multiplication
is defined as in \cite{VK}:
$$
(f\ast g)(s,t)=\sum_{r}f(s,r)g(r,t),
$$
where $(s,t)\in B(Y)$ and the sum is taken over all paths $r$ such that
$(s,r),(r,t)\in B(Y)$. We can also describe the decomposition of the algebra $A$
into simple components in terms of the graph $Y$. Denote by
$\tilde Y_{n}$ the set of orbits with respect to the action of
$\omega$ in $Y_{n}$.  Since $\omega $ is an involutive automorphism,
orbits can be of two types: those consisting of a single vertex
(type $Q$), and those consisting of two vertices (type $M$). Let
$t\in\tilde Y$, and let $n_{t},m_{t}$ be the numbers
of paths going from the two vertices of
$ Y_{0}$ to one of the vertices belonging to $t$. Note that in the case of
a vertex of type $Q$, we have $n_{t}=m_{t}$.
Then
 $$
 A=\oplus_{t\in\tilde Y_{n}}A_{t},
 $$
where $A_{t}=M(n_{t},m_{t})$ if the vertex is of type $M$, and
$A_{t}=Q(n_{t})$ if the vertex is of type $Q$.

\begin{definition} Let us say that two chains  $\frak Y,\frak Y^{\prime}$
of $\Bbb Z_{2}$-graded algebras are equivalent if for
$i=1,\dots, n$ there exist isomorphisms
$f_{i} : A(i)\rightarrow A(i)^{\prime}$ such that the diagrams
$$
\begin{array}{ccc}
A(i+1)&\stackrel{f_{i+1}}{\longrightarrow}&A^{\prime}(i+1) \\
\uparrow & &\uparrow \\
A(i)&\stackrel{f_{i}}{\longrightarrow}&A^{\prime}(i) \\
\end{array}
$$
are commutative for $i=1,\dots,n-1$.
\end{definition}

\begin{proposition}  Equivalence classes of chains are in a bijection
with the branching graphs of graded simple modules (regarded up to isomorphism).
\end{proposition}
\begin{proof}
It suffices to consider the case of a chain of length $2$. First note that
it suffices to restrict ourselves to chains of the form
$B\subset A$, $B\subset A^{\prime}$ and prove that the diagram
 \begin{equation}\label{diagram}
\begin{array}{ccc}
A &\stackrel{f}{\longrightarrow}&A^{\prime} \\
\uparrow & &\uparrow \\
B&\stackrel{{\rm id}_{B}}{\longrightarrow}&B \\
\end{array}
\end{equation}
is commutative if and only if the corresponding simple
$A,A^{\prime}$-modules  $V,V^{\prime}$ are isomorphic as $B$-modules.
Replacing the algebra $B$ by the projection to the corresponding simple component,
we may assume that $A,A^{\prime}$ are simple algebras. Besides, we
may replace the algebra $A^{\prime} $ by
$ A$. Thus it suffices to prove the following assertion. Let
$ A$ be a simple $\Bbb Z_{2}$-graded algebra, $V$ be a standard $A$-module,
$ B$ be a $\Bbb Z_{2}$-graded subalgebra of $A$, and
$\varphi$ be an automorphism of $ A$. Then $\varphi$ acts on
$ B$ identically if and only if it is of the form
$\varphi(a)v=faf^{-1}v$, where $f$ is an isomorphism either of the $B$-module $V$,
or of the $B$-modules $V$ and  $P(V)$. Let us prove this assertion. Assume that
$\varphi$ is of this form. Then for
$b\in B$ we obtain $\varphi(b)v=fbf^{-1}v=bv$.
Hence  $\varphi(b)=b$. Conversely, assume that
$\varphi(b)=b$ for every $b\in B$. The parity automorphism of the algebra $A$
is of the form $a\rightarrow JaJ^{-1}$, where $J$ is the parity automorphism
of the module  $V$. Since $\varphi$ preserves the parity,
$\varphi(J)=\pm J$. Further, there exists a linear map
$f : V\rightarrow V$ such that $\varphi(a)v=faf^{-1}v$. Since
$\varphi(J)=\pm J$, it follows that  $f$  is a graded homomorphism either
$V\rightarrow V$ or
$V\rightarrow P(V)$. The conditions
$\varphi(b)=b$, $\varphi(a)v=faf^{-1}v$ imply that $f$ is a homomorphism of
$B$-modules. The proposition is proved.
\end{proof}

A branching graph is called simple if it has no multiple edges. In order
to formulate a simplicity criterion, we use the notion of centralizer
from Section~2. The following lemma gives a simplicity criterion for
graded modules.

\begin{lemma}Let $B\subset A$ be a short chain of graded algebras. The branching
of the corresponding graded modules is simple if and only if
the algebra $Z(A,B)$ is commutative, and in this case it is generated by the
graded centers $Z(A),
Z(B)$. (Cf.\ the simplicity criterion in  \cite{OV,VO}).
\end{lemma}

\begin{proof}
It is not difficult to check that
$Z( A[\varepsilon], B[\varepsilon] )=Z_{0}(|A|,|B|)+\varepsilon Z_{0}^{\theta}(|A|,|B|)$.

Hence the algebra $Z(A,B) $ is the image of the algebra
$Z(A[\varepsilon], B[\varepsilon])$ under the
homomorphism that
sends $\varepsilon$ to  $1$. It easily follows that
$Z(A,B)$ is commutative if and only if so is
$Z( A[\varepsilon],B[\varepsilon] )$.
In this case, it is well known that $Z(A,B) $
is generated  by the subalgebras $Z(A[\varepsilon])$, $Z(B[\varepsilon] )$.
 \end{proof}

Like in the nongraded case, two types of problems arise for the graphs
under consideration. The analysis problem: given a graded algebra, construct
the graph of simple modules. And the synthesis problem: given a graph with
involution, describe the algebra for which it is the graph
of simple modules.

Let $G$ be a  ${\Bbb Z}_2$-graded group; then
$G_0$, being a subgroup of index $2$, is a normal subgroup of $G$, so that
$G$ is a ${\Bbb Z}_2$-extension of $G_0$. Recall (see
\cite{FH}) the relation between irreducible
representations of these groups.
The group ${\Bbb Z}_2$ acts on the simple modules of the group
$G_0$ and divides them into two classes: those that are fixed
under the action of ${\Bbb Z}_2$ and those that are not.
The direct sum of a module of the second class and its ${\Bbb Z}_2$-image
is an irreducible ${\Bbb Z}_2$-graded $G$-module.
Note that for ${\Bbb Z}_2$-graded algebras, the situation is similar but
more complicated. The existence of such relations implies
the existence of relations between the corresponding
branching graphs $Y(\frak Y)$, $Y(|\frak Y|)$, and $Y(\mathfrak Y_{0})$.

In contrast to the tradition, we will denote the alternating
group by $S_{n}^+$ rather than $A_{n}$.

Proposition~ \ref{eqv} implies the following lemma.

\begin{lemma}\label{eqv1} Assume that the odd part of every algebra
of a chain  $\frak Y$ of
$\Bbb Z_{2}$-graded algebras is nontrivial. Then the branching graph
  $Y({\mathfrak Y})$  of this chain  coincides with the branching graph
$Y(|\frak Y_{0}|)$.
\end{lemma}

\def\sgn{{\rm sgn}}

\begin{example} {\rm Consider the involutive automorphism
$\theta$ of the group algebra
$\mathbb C[S_{n}]$ of the symmetric group defined by the formula
$\theta(\sigma)=\sgn(\sigma)\sigma$, where $\sgn(\sigma)$ is the sign
of a permutation $\sigma$, and the corresponding structure of a graded algebra
on $\mathbb C[S_{n}]$. Let
$Y(\mathbb S)$ be the graded branching graph of the chain
$$
\mathbb S=< \mathbb C=  \mathbb C[S_{1}]\subset\mathbb C[S_{2}]\subset\dots \subset \mathbb C[S_{n}]>
$$
of the group algebras of the symmetric groups with this
$\Bbb Z_{2}$-grading, and let $Y(|\mathbb S^+|)$
be the nongraded branching graph of the chain
$$
\mathbb S^+=<\mathbb C=  \mathbb C[S^{+}_{1}]\subset\mathbb C[S^{+}_{2}]\subset\dots \subset \mathbb C[S_{n}^+]>
$$
 of the group algebras of the alternating groups.
Then it follows from Lemma~\ref{eqv1} that the graph  $Y(\mathbb S)$
coincides with the graph
$Y(|\mathbb{S^+}|)$ at all levels except the first one.}
\end{example}

Note that an irreducible graded representation of the
symmetric group coincides with the ordinary representation corresponding
to a diagram $\lambda$ if
$\lambda=\lambda^{\prime}$, and is equal to the direct sum of the ordinary
representations corresponding to
the diagrams $\lambda,\lambda^{\prime}$ if $\lambda\ne\lambda^{\prime}$.

Consider the involutive automorphism
$\theta$ of the group algebra
$\mathbb C[S^+_{n}]$  of the alternating group given by the formula
$\theta(f)=s_{12}fs_{12}$, where $s_{12}$ is the transposition of the symbols
$1,2$, and the corresponding structure of a superalgebra on
$\mathbb C[S^+_{n}]$; then Lemma~\ref{epsilon} implies the following theorem.

\begin{thm}
{\rm1)} An ordinary irreducible representation of the group
$S_{n}$ remains irreducible regarded as a graded representation of the group
$A_{n}$. It is of type $Q $ if the corresponding Young diagram
$\lambda$ is self-conjugate, i.e., $\lambda=\lambda^{\prime}$; it
is of type $ M$ if the corresponding Young diagram
$\lambda$ is not self-conjugate, i.e., $\lambda\ne\lambda^{\prime}$.

{\rm2)} Let $Y(\mathbb S^+)$ be the graded branching graph of the chain
$$
\mathbb S^+=< \mathbb C=  \mathbb C[S^+_{1}]\subset\mathbb C[S^+_{2}]\subset\dots \subset \mathbb C[S^+_{n}]>
$$
of the group algebras of the alternating groups with the
$\Bbb Z_{2}$-grading introduced above,
and let $Y(|\mathbb S|)$ be the nongraded branching graph of the chain
$$
\mathbb S=<\mathbb C=  \mathbb C[S_{1}]\subset\mathbb C[S_{2}]\subset\dots \subset \mathbb C[S_{n}]>$$
 of the group algebras of the symmetric groups (i.e., the Young graph).
Then
$Y(\mathbb S^+)$ coincides with $Y(|\mathbb{S}|)$
at all levels except the first one.
\end{thm}

\section{Gelfand--Tsetlin algebras}

Consider a chain of $\Bbb Z_{2}$-graded algebras
 $$
\frak Y=< \mathbb C= A(1)\subset A(2)\subset\dots \subset A(n)>.
 $$

\begin{definition} Assume that the branching graph of the chain
$\frak Y$ is simple. The Gelfand--Tsetlin algebra
$GZ(\frak Y)$ is the algebra generated by the graded centers
$Z(A_{i})$, $i=1,\dots,n$.
\end{definition}

\begin{thm}
Consider a chain of $\Bbb Z_{2}$-graded algebras
 $$
\frak Y=< \mathbb C= A(1)\subset A(2)\subset\dots \subset A(n)>
 $$
with simple branching. Then the Gelfand--Tsetlin algebra
$GZ(\frak Y)$ coincides with the ordinary Gelfand--Tsetlin algebra $GZ(\frak Y_{0})$
of the chain of even subalgebras. It is a maximal commutative subalgebra among
the even commutative subalgebras in
$A(n)$. Every irreducible graded representation of the algebra
$A(n)$ has a homogeneous basis that consists of eigenvectors of this subalgebra.
\end{thm}

\begin{proof}  Consider the chain of algebras
  $$
\mathbb C\subset \mathbb C[\varepsilon]= A(1)[\varepsilon]\subset A(2)[\varepsilon]\subset\dots \subset A(n)[\varepsilon].
 $$
Since the branching is simple, the algebra generated by
$Z(A(i+1)[\varepsilon], A(i)[\varepsilon])$,
$i=1,\dots, n-1$, is a maximal commutative subalgebra in
$A(n)[\varepsilon]$. Further, consider the homomorphism
$A(n)_{0}[\varepsilon] \rightarrow A(n)_{0} $ that sends
$\varepsilon$ to the identity. The image of every maximal
commutative subalgebra in
$A(n)_{0}[\varepsilon]$ is a maximal
commutative subalgebra in $A(n)_{0}$.
 \end{proof}

Thus, to a certain extent, the representation theory of a chain of
$\Bbb Z_{2}$-graded algebras reduces to studying representations of the
chain of the even parts of these algebras.
However, it is not always possible to describe the even part of a graded
algebra in a transparent way, so that it is useful
to introduce another version of the notion
of Gelfand--Tsetlin algebra, which is related to the notion
of supercentralizer (see Section~2) used in the theory of Lie superalgebras.

\begin{definition}Consider a chain of $\Bbb Z_{2}$-graded algebras
 $$
\frak Y=< \mathbb C= A(1)\subset A(2)\subset\dots \subset A(n)>.
 $$

{\rm1)} The algebra
 $$
SGZ(\frak Y)=<SZ(A(2),A(1)),\dots,SZ(A(n),A(n-1))>
$$
generated by the successive supercentralizers  is called the Gelfand--Tsetlin superalgebra.

{\rm 2)} Denote by $SZ(\frak Y)$ the commutative algebra generated by the
supercenters $SZ(A(1)),\dots,SZ(A(n))$.
\end{definition}

As will be shown below, the algebra
$SZ(\frak Y)$ is the supercenter of the algebra $SGZ(\frak Y)$.

In the nongraded case,
for chains with simple branching
the algebra generated by the successive centralizers coincides with the algebra
generated by the centers: $SZ(\frak Y)= SGZ(\frak Y)$.  If the
spectrum is not simple,
$SGZ(\frak Y) $ is the so-called ``big Gelfand--Tsetlin algebra.''
In the graded case, we have the following (in general, strict) inclusions
$$
SZ(\frak Y)\subset GZ(\frak Y)\subset SGZ(\frak Y),
$$
the first two algebras being commutative. The role of the Gelfand--Tsetlin
superalgebra is that the decompositions of irreducible
$A(n) $-modules into irreducible
$SGZ(\frak Y) $-modules are disjoint.
In the case of trivial grading
and simple spectrum, the algebra
$SGZ(\frak Y) $ coincides with $SZ(\frak Y)$.
The main result of this section is the following theorem.

\begin{thm}\label{SC2}
Consider a chain of $\Bbb Z_{2}$-graded algebras
 $$
 \frak Y=<\mathbb C= A(1)\subset A(2)\subset\dots \subset A(n)>.
 $$
Then

{\rm1)} The restriction of an irreducible
$A(n)$-module to the Gelfand--Tsetlin superalgebra has a simple spectrum.

{\rm2)} The decompositions of different irreducible  $A(n)$-modules into
irreducible $SGZ(\frak Y)$-modules have no common components. The irreducibility
type of a nonzero
$SGZ(\frak Y)$-module coincides with the irreducibility type of the
$A(n)$-module that contains it.
\end{thm}

 \begin{proof} The proof is based on a series of lemmas.

\begin{lemma}\label{dct} Let
$V$ be a finite-dimensional graded vector space,
$A\subset \End(V)$ be a semisimple graded subalgebra of $\End(V)$,
and
$A^{\prime}$ be its supercentralizer. Then
$(A^{\prime})^{\prime}=A$.
\end{lemma}

The proof of the lemma
is contained in \cite{NS}. It is similar to von Neumann's theorem
on the bicommutant of subalgebras in simple algebras, and can be
proved analogously.

\begin{lemma}\label{simple}
Let $A $ be a $\Bbb Z_{2} $-graded algebra and $ B$ be a
$\Bbb Z_{2}$-graded subalgebra of $A$. Then the equality
$SZ(A,B)=SZ(B)$  holds if and only if every irreducible $A$-module
decomposes into a multiplicity-free sum of simple $B$-modules
and different irreducible $A$-modules have no common irreducible $B$-components.
Moreover, every irreducible $B$-component has the same irreducibility
type as the irreducible $A$-module that contains it.
\end{lemma}

\begin{proof}
Assume that every irreducible $A$-module decomposes
into a multiplicity-free sum of simple $B$-modules
and different irreducible $A$-modules have no common irreducible $B$-components.
Let
$V_{1},\dots,V_{q}$ be the complete set of pairwise nonisomorphic
irreducible $A$-modules,
$U_{1},\dots, U_{p}$  be all different $B$-modules, and
$b_{1},\dots, b_{p}$ be the corresponding central idempotents. Let
$a\in SZ(A,B)$.  By the assumptions of the lemma, $a$ is an even element.
Since every irreducible $A$-module decomposes
into a multiplicity-free sum of simple $B$-modules, the element
$a$ preserves each module $U_{j}$  and acts in it as a scalar
$c_{j}$.  Since different irreducible $A$-modules have no common irreducible $B$-components,
the difference $a-c_{1}b_{1}-\dots-c_{p}b_{p}$ vanishes in every irreducible
$A$-module. This proves that
$SZ(A,B)=SZ(B)$.

Conversely, assume that $SZ(A,B)=SZ(B)$. Therefore
$SZ(A)\subset SZ(B)$. Let $e_{1},e_{2}$
be two different central idempotents in
$A$ corresponding to simple $A $-modules $V_{1},V_{2}$; then
$$
e_{1}=c_{1}b_{1}\dots+c_{p}b_{p},\quad e_{2}=d_{1}b_{1}+\dots+d_{p}b_{p}.
$$
The condition $e_{1}e_{2}=0$ implies that $c_{i}d_{i}=0$ for $i=1,\dots,p$.
It follows that the modules
$V_{1},V_{2}$ have no common irreducible $B$-components. Moreover,
the equalities
$e_{1}^2=e_{1}$, $e_{2}^2=e_{2}$ imply that
$c_{i},d_{i}=0,1$. This proves that
every irreducible $A$-module decomposes
into a multiplicity-free sum of simple $B$-modules.
The claim concerning the irreducibility type follows from the fact that
$SZ(A,B)$ contains only even elements.
\end{proof}

 \begin{lemma}\label{quotient} Let  $f: A\longrightarrow B$
 be a homomorphism of semisimple superalgebras and $C$
be a simple subsuperalgebra in $A$. Then
$$
f(SZ(A,C))=SZ(f(A),f(C)).
$$
\end{lemma}

\begin{proof} Clearly, $f(SZ(A,C))\subset SZ(f(A),f(C))$.
Let us prove the converse inclusion.
Let $b=f(a)\in SZ(f(A),f(C))$. Since $A $
is a semisimple superalgebra,
$\ker f=eA$, where
$e$ is a central idempotent. We have
$ ac-(-1)^{p(a)p(c)}ca\in eA$ for every $c\in C$. Hence
$a(1-e)c=(-1)^{p(a)p(b)}c(1-e)a$, also for every $c\in C$. Therefore
$a(1-e)\in SZ(A,C)$. Hence
 $f(a)=f(a(1-e))$.
\end{proof}

\begin{lemma}\label{chain}
Let $A$ be a semisimple graded algebra, $B$ be a semisimple graded
subalgebra of $A$, and $C$ be a graded subalgebra in $A$ generated by
$SZ(A,B)$ and $B$. Given an irreducible $A$-module $V$ and an irreducible
$B$-module $U$, denote by $I_{U}(V)$ the sum of all $B$-submodules in $V$
isomorphic to $U$ or $P(U)$. Then
$I_{U}(V)$ is an irreducible $C$-module, and every irreducible $C$-module
is of this form. Besides, two nonzero modules of this form
are isomorphic if and only if
$U=U^{\prime}$, $V=V^{\prime}$.
\end{lemma}

\begin{proof}
By Lemma~\ref{quotient}, we can replace the algebras $ A$ and
$ B$ by their images in $\End(V)$.
Clearly,
$I_{U}(V)$ is a $C$-module. Besides, the algebras
$SZ(A,B)$ and $ B$ are semisimple. Hence their tensor product is also
a semisimple algebra, and $C$, being the
homomorphic image of this tensor product,
is also semisimple. Further, we have a canonical decomposition
$V=W\oplus I_{U}(V)$, where $W$ is the sum of all $B$-submodules in
$V$ that are not isomorphic to $U$ or $P(U)$. Thus we have a canonical
homomorphism $\End(V)\rightarrow \End( I_{U}(V))$. Hence, in order to prove the
irreducibility of  $I_{U}(V)$, it suffices to compute the supercentralizer of $C$
in $I_{U}(V)$. Further, $V$ is an irreducible $A$-module, hence there are
two possibilities: $\End(V)=A$ and $\End(V)=A[\pi]$,
where $\pi$ is an odd element from $\End(V)$ that supercommutes with
$A$. In the first case, by Lemma~\ref{dct}, the supercentralizer of
$C$ in $\End(V)$  coincides with the center of $B$. Hence, by Lemma~\ref{quotient},
the centralizer of $C$ in $I_{U}(V)$ coincides with $\mathbb{C}$. In the second case,
the same arguments show that the supercentralizer of $C$ in
$I_{U}(V)$ coincides with $\mathbb{C}[\pi]$.

Conversely, let us prove that every
irreducible module is of this form. Let $W$ be an irreducible $C$-module and
$U$ be an irreducible $B$-submodule of $W$. Consider the induced module
$A\otimes_{C}W$ and the irreducible component $V$ of this module that contains $W$.
Then, by the above, the module $I_{U}(V) $ is irreducible and hence coincides
with $W $.

It remains to prove only that the condition
 $I_{U}(V)=I_{U^{\prime}}(V^{\prime})$ implies
$(U,V)=(U^{\prime}, V^{\prime})$. This assertion is equivalent to the fact
that different irreducible $A$-modules have no common irreducible $C$-components.
By Lemma~\ref{simple}, this is equivalent to the condition
$SZ(A,C)=SZ(C)$.  But it is easy to see from the definition that
$SZ(A,C)=SZ(B^{\prime})$, where $B^{\prime}=SZ(A,B)$. Hence it suffices to show that
$SZ(B^{\prime})=SZ(C)$. But the algebra $C$ is semisimple, and it is a homomorphic
image of the algebra $B^{\prime}\otimes B$, which is also semisimple
with supercenter equal to  $SZ(B)^{\prime}\otimes SZ(B)$. Hence the supercenter
of $C$ is equal to the product of the supercenters of  $B^{\prime}$ and $B$.
But it is easy to see that $SZ(B)\subset SZ(B^{\prime})$, whence $SZ(C)=SZ(B^{\prime})$.
\end{proof}

\begin{corollary}Consider a chain of
$\Bbb Z_{2}$-graded semisimple finite-dimensional algebras
$$
 \frak Y=<A(1)\subset\dots\subset A(n-1) \subset A(n)>
$$
(we do not assume that $A(1)=\mathbb {C}$).
Also let $V_{1},\dots,V_{n-1},V_{n}$ be irreducible modules over the algebras
$A(1),\dots,A(n-1),A(n)$, respectively, and
$ SGZ(\frak Y)$ be the algebra generated by the supercentralizers
$SZ(A(i+1),A(i))$, $i=1,\dots, n-1$, and the algebra
$A(1)$. Then the algebra
$ SGZ(\frak Y)$
is semisimple.

Denote by $I(V_{1},V_{2},\dots, V_{n})$
the subspace $I_{V_{1}}(\dots( I_{V_{n-1}}(V_{n})\dots) )$.
 Then the nonzero subspaces
$$
I(V_{1},V_{2},\dots, V_{n})
$$
are irreducible $ SGZ(\frak Y)$-modules, and every irreducible
$ SGZ(\frak Y)$-module is of this form. Moreover, if
 $$
  I(V_{1},\dots ,V_{n-1},V_{n} )= I(V^{\prime}_{1},\dots,V^{\prime}_{n-1},V^{\prime}_{n} ),
 $$
then $V_{1}=V_{1}^{\prime},\dots ,V_{n}=V_{n}^{\prime}$.
\end{corollary}

\begin{proof}
Induction on $n$. If $n=2$, then this is the assertion of the previous lemma.
Let $n>2$. Denote by $C$ the subalgebra in
$A(n) $ generated by  $SZ(A(i+1),A(i))$, $i=2,\dots, n-1$, and by $B$
the algebra generated by $A(2)$ and $C $. Consider the chain
$A(1)\subset B$. By the induction hypothesis, the algebra $C$ is semisimple,
hence the tensor product $A(2)\otimes C $ is also semisimple. Further, it is
easy to check that the supercentralizer of the algebra
$A(1)$ in the algebra $A(2)\otimes C
$ equals $SZ(A(2), A(1))\otimes C$. Since the algebra $B$
is a homomorphic image of the algebra
$A(2)\otimes C $, it follows from Lemma~\ref{quotient} that
$SZ(B,A(1))$ equals $ SGZ(\frak Y)$. Now we can use Lemma~\ref {chain}.
\end{proof}

Now let us prove the theorem. To this end, we first prove that the supercenter
of the algebra $ SGZ(\frak Y)$ is generated by the supercenters of the algebras
$A(1),\dots, A(n)$. Let $W$ be an irreducible $ SGZ(\frak Y)$-module and $z_{W}$
be the
element of the supercenter that acts as
$1$ in the modules $W,\: P(W)$ and acts as $0$ in all the other irreducible modules.
By the previous lemma,
$W= I_{V_{1}}(\dots( I_{V_{n-1}}(V_{n})\dots) )$. Let
$z_{V_{i}}$, $i=1,\dots, n$, be the elements similar to $z_{W}$.
Then the difference $z_W-z_{V_1}{\ldots} z_{V_n}$
acts as $0$ in every irreducible
$ SGZ(\frak Y)$-module and hence vanishes. This proves that the supercenter
of $ SGZ(\frak Y)$ is generated by the supercenters of $A(1),\dots,A(n)$.
Since irreducible representations of
$ SGZ(\frak Y)$, regarded  up to
the functor $P$,
are in a bijection with homomorphisms of
the supercenter of $ SGZ(\frak Y)$,
the theorem follows.
 \end{proof}

In the case of a simple branching,
a description of subspaces
of irreducible modules over the Gelfand--Tsetlin superalgebra
can be given in terms of the corresponding branching graph.

\begin{definition} Assume that the branching graph of a chain
$\frak A$ is simple. In this case, for every path
$$
T=\lambda_{1}\nearrow\dots\nearrow\lambda_{i}\nearrow\lambda_{i+1}
\nearrow\dots\nearrow\lambda_{n}
$$
in the branching graph there is a vector
$v_{T}$, unique up to a nonzero constant, in the corresponding
irreducible representation. This vector is called the Gelfand--Tsetlin vector.
The set of Gelfand--Tsetlin vectors forms a basis, which is called
the Gelfand--Tsetlin basis.
\end{definition}

\begin{definition} Two vectors $v,w$ of the Gelfand--Tsetlin basis
are called equivalent if for the corresponding paths
$$
v\subset A(1)v\subset A(2)v\subset\dots\subset A(n)v,\quad w\subset
A(1)w\subset A(2)w\subset\dots\subset A(n)w
$$
in the branching graph, the following equations hold:
$ A(n)w= A(n)v$, and for every $1\le i\le n-1$,
either $A(i)w=A(i)v$ or $P(A(i)w)=A(i)v$.
\end{definition}

\begin{lemma} Equivalent vectors of the Gelfand--Tsetlin basis
determine the same homomorphism $\chi : SZ(\frak A)\rightarrow\mathbb{C}$
of the supercenter of the algebra
$SGZ(\frak A)$. The linear span of these equivalent vectors
coincides with $V_{\chi}$.
\end{lemma}

\begin{proof}
 Let $v,w$ be equivalent Gelfand--Tsetlin vectors
 and  $z\in Z(A(i))$ for $1\le i\le n$. Since the irreducible modules
$A(i)v$ and $A(i)w$ differ only by the parity, $z$ acts on $v$ and $w$ by the
same scalar. Therefore the vectors $v,w$ determine the same homomorphism
$\chi : SZ(\mathfrak Y)\rightarrow \mathbb C$. Now let us prove the
converse.
It is enough to prove the  following claim: if $a$ acts as $1$ on every
vector of the Gelfand--Tsetlin basis that is equivalent to $v$ and acts as $0$
on every nonequivalent vector, then $a\in SZ(\mathfrak Y)$.  Let
$$
v\subset A(1)v\subset A(2)v\subset\dots\subset A(n)v
$$
be the path corresponding to the vector $v$. Since the branching is simple,
we can find $z_{1},\dots,z_{i},\dots,z_{n}$ such that $z_{i}$ acts as $1$ in
$A(i)v,\:P(A(i)v)$ and acts as $0$ in every irreducible module nonisomorphic to $A(i)v,\: P(A(i)v)$. Then the product $z_{1}\dots z_{n}$ acts in the same way as $a$ in every irreducible $A(n)$ module. Therefore $a=z_{1}\dots z_{n}$.\end{proof}

\begin{corollary} If the branching of a chain of graded algebras is simple,
then the corresponding Gelfand--Tsetlin algebra is the direct sum
of Clifford algebras.
\end{corollary}

\section{Projective representations of the symmetric groups}

I.~Schur \cite{Shc} proved that the symmetric group $S_n$ has only one
nontrivial central extension and suggested a method for finding projective
representations of $S_n$. This central extension,
which we will denote by  $\tilde{S_n}$, is  a
$\Bbb Z_2$-graded group in the sense defined above. However,
the corresponding grading of the group algebra of $\tilde{S_n}$
is more complicated than in the
case of ${\Bbb C}[S_n]$ considered above.
Proper
projective
representations of the symmetric group coincide with representations
of some ${\Bbb Z}_2$-graded algebra $\frak A_{n}$, which is described below and
which is the ``half'' of the group algebra of
$\tilde{S_n}$, more exactly, the quotient of ${\Bbb C}[\tilde{S_n}]$
modulo the ``half'' ideal.
The second ideal, complementary to the first one,
is the image of the group algebra
${\Bbb C}[S_n]$ under the (non-identity-preserving) embedding,
and the quotient of ${\Bbb C}[\tilde{S_n}]$ modulo this ideal coincides with
${\Bbb C}[S_n]$. Here both quotient algebras, ${\Bbb C}[S_n]$ and $\frak A_{n}$,
inherit the $\Bbb Z_2$-grading of the algebra ${\Bbb C}[\tilde{S_n}]$.
It follows from the above considerations that representations of every
${\Bbb Z}_2$-graded algebra with an essential grading are in a bijective
correspondence with representation of its even part. However, the even part
of $\frak A_{n}$ still has no convenient model. Hence, in order
to describe its representations, we use the notions
of graded representations of the Gelfand--Tsetlin superalgebra $SGZ(\frak Y)$
introduced above.
We reproduce the method of constructing the representation theory
of the symmetric groups suggested in \cite{OV,VO}. Namely, we construct analogs
of the YJM-elements, spectra, etc., and finally obtain the known list
of projective representations of the symmetric groups.

We consider two algebras related to projective representations
of the symmetric group.
The first one is the algebra $\frak A_{n}$ described below,
and the second one is the (graded) tensor product
of the Clifford algebra $\mathcal{C}_{n}$ with $n$ generators and
the algebra $\frak{A}_{n}$. For $\frak A_{n}$, Young's orthogonal form
was constructed in \cite{N1}; and for
$\mathcal{C}_{n}\otimes\frak{A}_{n}$, Young's orthogonal and seminormal
forms were constructed in \cite{N2}. We construct seminormal forms
for both algebras. It turns out that for
$\mathcal{C}_{n}\otimes\frak{A}_{n}$ Young's formulas look simpler and
coincide with the formulas obtained by M.~Nazarov in
 \cite{N2} by other methods.

Thus graded representations of $\frak{A}_{n}$ are in a one-to-one
correspondence with nongraded proper projective representations
of the alternating group, or, which is the same,
with representations of a central extension of the alternating group
in which the central element acts in a nontrivial way.

Let us proceed to a more detailed description of the groups and algebras
under consideration.
The group $\tilde S_n$ is given by generators $\tilde \sigma_k$,
$k=0,1,\dots ,n-1$, satisfying the  relations
$$\tilde{\sigma_k}^2=1,\quad \tilde\sigma_k \tilde\sigma_0=\tilde \sigma_0\tilde\sigma_k, \quad
 (\tilde \sigma_k\tilde \sigma_{k+1})^3=1, \quad k=0,1,\dots ,n-1,$$
$$\tilde\sigma_i\tilde\sigma_k=\tilde\sigma_0\tilde\sigma_k\tilde\sigma_i,
\quad i,k=0,1,\dots ,n-1.$$

The grading is defined as follows: the identity element is assumed to be even,
and all the elements $\tilde\sigma_k$, $k>0$, are assumed to be odd. Note that
$\tilde\sigma_0$ is even. The quotient
of the group $\tilde S_n$ over its center
${\Bbb Z}_2=\{\tilde \sigma_0\}$ is the symmetric group, so that
$\tilde{S_n}$ is a ${\Bbb Z}_2$-extension of $S_n$,
but this extension is not trivial. The group algebra
${\Bbb C}[{\tilde S}_n]$ decomposes into the sum
of two ideals. The first one is generated by the relation
$\tilde\sigma_0-1=0$, and the quotient modulo this ideal is the group algebra
of the symmetric group with the ordinary parity grading. The second
one is generated by the relation $\tilde\sigma_0+1=0$
and does not corresponds to any normal subgroup
of $\tilde{S_n}$; the quotient modulo this ideal is, by definition,
a graded algebra $\frak{A}_{n}$, which is not a group algebra.
Representations of this algebra are exactly proper projective representations
of the symmetric group. Note also that the normal subgroup that is the
even part of $\tilde{S_n}$ is the restriction of the
central extension of
$S_n$ to the alternating group $S_n^+$, i.e., a ${\Bbb Z}_2$-extension of the
alternating group; this extension is also nontrivial.

Thus the algebra $\frak{A}_{n}$ is generated by elements
$\tau_{1},\dots,\tau_{n}$ that are the images of the generators
$\tilde\sigma_k$, $k>0$, of   $\tilde{S_n}$ and satisfy the relations
\begin{equation}\label{relation}
\tau_{k}^2=1,\quad(\tau_{k}\tau_{k+1})^3=1,\quad(\tau_{k}\tau_{l})^2=-1\;
\text{ if }\;\mid k-l\mid >1.
\end{equation}
The algebra  $\frak{A}_{n}$ has a natural grading inherited from
${\Bbb C}[{\tilde
S}_n]$: $p(\tau_{i})=1$, $i=1,\dots, n-1$.
We can rewrite the defining relations in the form
\begin{equation}\label{relation1}
 \tau_{k}^2=1,\quad
 \tau_{k}\tau_{k+1}\tau_{k}=\tau_{k+1}\tau_{k}\tau_{k+1},\quad
 \tau_{k}\tau_{l}=-\tau_{l}\tau_{k}\;\text{ if }\;\mid k-l\mid >1.
\end{equation}
In fact, it is more convenient to describe the algebra
 $\frak A_{n}$ in a slightly different way. Let us define elements
$\tau_{ij}$ for $i<j$ by induction as follows:
$$
  \tau_{ii+1}=\tau_{i},\quad
  \tau_{ij}=-\tau_{ij-1}\tau_{j}\tau_{ij-1},\quad
  \tau_{ji}=-\tau_{ij},
$$
where $i=1,\dots,n-1$.  It is easy to check the following relations for
$i,j\in\{1,\dots,n\}$ and $i\ne j$:
$$
  \tau_{ij}=-\tau_{ji},\quad\tau_{ij}^2=1,\quad
  \tau_{ij}\tau_{kl}=-\tau_{kl}\tau_{ij}\;\text{ if }\;\{i,j\}\cap\{k,l\}=\emptyset,
$$
$$
\tau_{ij}\tau_{jk}\tau_{ij}=\tau_{jk}\tau_{ij}\tau_{jk}=-\tau_{ij},
$$
where  $i,j,k$ are pairwise distinct.
We will regard
$\frak A_{n}$ as a superalgebra, setting $p(\tau_{ij})=1$.
Note that we can obtain a similar description of the algebra
 $\frak A^{+}_{n}$. To this end, recall the following result from
\cite{VV}. The group $S_{n}^+$ is generated by elements
$x_{i}=(i,i+1, i+2)$, $i=1,\dots,n-2$, satisfying the relations
$$
x_{i}^3=1,\quad i=1,\dots, n-2,\:
$$
$$
(x_{i}x_{i+1})^2=1,\quad i=1,\dots, n-3,
$$
$$
x_{i}x_{j}=x_{j}x_{i},\quad |i-j|>2,\: i,j=i=1,\dots, n-2,
$$
$$
x_{i}x_{i+1}^{-1}x_{i+2}=x_{i+2}x_{i},\quad i=1,\dots, n-4.
$$
This easily implies the following proposition.

\begin{proposition} The algebra $\frak A^{+}_{n}$ is generated by elements
$y_{i}$, $i=1,\dots,n-2$, satisfying the relations
$$
y_{i}^3=1,\quad i=1,\dots, n-2,
$$
$$
(y_{i}y_{i+1})^2=-1,\quad i=1,\dots, n-3,
$$
$$
y_{i}y_{j}=y_{j}y_{i},\quad |i-j|>2,\: i,j=i=1,\dots, n-2,
$$
$$
y_{i}y_{i+1}^{-1}y_{i+2}=-y_{i+2}y_{i},\quad i=1,\dots, n-4.
$$
\end{proposition}

Now let us describe the supercenter (see Definition~\ref{super}) of the algebra $\frak A_{n}$.

\begin{thm} Let $\alpha \vdash n$ be a partition of $n$ such that all
parts of $\alpha$ are odd. The elements
$C_{n}^{\alpha}$ with such $\alpha$
form a basis of the supercenter of the algebra $\frak A_{n}$.
\end{thm}

\begin{proof}
Let $i_{1},\dots,i_{p}$ be a sequence of pairwise distinct elements from $\{1,2,\dots,n\}$. Set
$$
\tau_{i_{1}i_{2}\dots i_{p}}=\tau_{i_{1}i_{2}}\tau_{i_{2}i_{3}}\dots\tau_{i_{p-1}i_{p}}.
$$
Then it is easy to check that
$\tau_{i_{1}i_{2}\dots i_{p}}=-\tau_{i_{2}i_{3}\dots i_{p}i_{1}}$
for an even $p$, and
$\tau_{i_{1}i_{2}\dots i_{p}}=\tau_{i_{2}i_{3}\dots i_{p}i_{1}}$ for an odd $p$.

Let $\alpha=(\alpha_{1},\alpha_{2},\dots,\alpha_{s})$
be a partition of $m$, where $m\le n$ and
$\alpha_{1}\ge\alpha_{2}\ge\dots\ge\alpha_{s}\ge2$. Set
$$
C_{n}^{\alpha}=\sum\tau_{i^{(1)}_{1}\dots
i^{(1)}_{\alpha_{1}}}\tau_{i^{(2)}_{1}\dots
i^{(2)}_{\alpha_{2}}}\dots\tau_{i^{(s)}_{1}\dots
i^{(s)}_{\alpha_{s}}},
$$
where the sum is taken over all  pairwise distinct $i^{(r)}_{j}$ from
$\{1,2,\dots, n\}$.

Define the following action of the symmetric group
  $S_{n}$ on $\frak A_{n}$:
$$
\sigma_{ij}(\tau_{kl})=-\tau_{ij}\tau_{kl}\tau_{ij},
$$
where $\sigma_{ij}\in S_{n}$ is a transposition. We see that in order to
describe the supercenter, we must describe invariant elements with
respect to this action.
The supercenter is the linear span of the elements
$C_{n}^{\alpha}$ with $\alpha$ an arbitrary partition of
 $n$. Let us show that if $\alpha_{j}$ is even for some $j$, then
$C_{n}^{\alpha}$ vanishes. Indeed,
$$
C_{n}^{\alpha}=\sum\tau_{i^{(1)}_{1}\dots
i^{(1)}_{\alpha_{1}}}\dots\tau_{i^{(j)}_{1}\dots
i^{(j)}_{\alpha_{j}}}\dots\tau_{i^{(s)}_{1}\dots
i^{(s)}_{\alpha_{s}}}
$$
$$
=\sum\tau_{i^{(1)}_{1}\dots
i^{(1)}_{\alpha_{1}}}\dots\tau_{i^{(j)}_{2}\dots
i^{(j)}_{\alpha_{j}}i^{(j)}_{1}}\dots\tau_{i^{(s)}_{1}\dots
i^{(s)}_{\alpha_{s}}}=-C_{n}^{\alpha}.
$$
Therefore  $C_{n}^{\alpha}=0$. If all parts of
$\alpha$ are odd, then $C_{n}^{\alpha}\ne0$. Note that for different
$\alpha^{(1)},\dots,\alpha^{(q)}$ the corresponding elements
$C_{n}^{\alpha^{(1)}},\dots,C_{n}^{\alpha^{(q)}}$
have no common summands and hence are linearly independent.
\end{proof}

Now let us describe the supercentralizer
$SZ(\frak A_{n},\frak A_{n-1})$. To this end, we introduce
the following analogs of the YJM-elements.

\begin{definition} \cite{Se}
The Young--Jucys--Murphy (YJM) element is the following element
of the algebra
$\frak A_{n}$:
$$
\pi_{n}=\tau_{1n}+\tau_{2n}+\dots+\tau_{n-1n}.
$$
\end{definition}

Note that $\pi_{1}=0$ by definition. Note also that analogs of the YJM-elements
for the algebra $\mathcal{C}_{n}\otimes\frak{A}_{n}$
were introduced by M.~Nazarov in \cite{N2} and denoted by $x_n$;
they coincide, up to an element of the Clifford algebra, with
$\pi_{n}$, namely,
$x_{n}=\frac{1}{\sqrt{2}}p_{n}\pi_{n}$.
It is known that the elements
$x_{i}$, $i=1,\dots, n$, pairwise commute. It follows that the elements
$\pi_{i}$, $i=1,\dots, n$, pairwise anti-commute. Below we will give
an independent proof of this fact.

\begin{thm}\label{SC} The supercentralizer
$SZ(\frak A_{n},\frak A_{n-1})$ is generated by the supercenter
$SZ(\frak A_{n-1})$ and the element $\pi_{n}$.
The Gelfand--Tsetlin superalgebra coincides with the algebra $\mathbb C
[\pi_{1},\pi_{2},\dots,\pi_{n}]$, and the algebra $SZ(\frak Y)$
generated by the supercenters
coincides with
$ \mathbb C [\pi_{1}^2,\pi_{2}^2,\dots,\pi_{n}^2]$.
\end{thm}

\begin{proof} As in the previous theorem, we must average elements of
$\frak A_{n}$ with respect to the action of
$S_{n-1}$. Set
$$
C_{n}^{\alpha,p}=\sum\tau_{i_{1}\dots
i_{p}n}\tau_{i^{(1)}_{1}\dots
i^{(1)}_{\alpha_{1}}}\dots\tau_{i^{(j)}_{1}\dots
i^{(j)}_{\alpha_{j}}}\dots\tau_{i^{(s)}_{1}\dots
i^{(s)}_{\alpha_{s}}},
$$
where the sum is taken over all pairwise distinct
$i_{l},i^{j}_{r}$ from  $\{1,2,\dots,n-1\}$. It is not difficult to check that
$SZ(\frak A_{n},\frak A_{n-1})$ is the linear span of $C_{n}^{\alpha,p}$ with
$\alpha_{1}+\dots+\alpha_{s}+p\le n$. As in the proof of the previous theorem, we
see that all parts of $\alpha$ are odd and $p$ is arbitrary. Let us use
the induction on
 $\alpha_{1}+\dots+\alpha_{s}+p$. If $p=0$, then
 $C_{n}^{\alpha,0}\in SZ(\frak A_{n-1})$. Hence we may assume that
  $p>0$. Consider the product
$$
\pi_{n}C_{n}^{\alpha,p-1}=\sum\tau_{in}\tau_{i_{1}\dots
i_{p-1}n}\tau_{i^{(1)}_{1}\dots
i^{(1)}_{\alpha_{1}}}\dots\tau_{i^{(s)}_{1}\dots
i^{(s)}_{\alpha_{s}}}
$$
$$
=C_{n}^{\alpha,p}+\sum_{i=i_{1},\dots,i_{p-1}}\tau_{in}\tau_{i_{1}\dots
i_{p-1}n}\tau_{i^{(1)}_{1}\dots
i^{(1)}_{\alpha_{1}}}\dots\tau_{i^{(s)}_{1}\dots
i^{(s)}_{\alpha_{s}}}
$$
$$
+\sum_{j=1}^s\sum_{i=i^{(j)}_{1},\dots,i^{(j)}_{\alpha_{j}}}
\tau_{in}\tau_{i_{1}\dots i_{p-1}n}\tau_{i^{(1)}_{1}\dots
i^{(1)}_{\alpha_{1}}}\dots\tau_{i^{(s)}_{1}\dots
i^{(s)}_{\alpha_{s}}}.
$$
It is easy to check that
$$
\tau_{i_{q}n}\tau_{i_{1}\dots i_{p-1}n}= (-1)^p\tau_{i_{1}\dots
i_{q-1}n}\tau_{i_{q}\dots i_{p-1}n}
$$
and
$$
\tau_{i^{(j)}_{1}n}\tau_{i_{1}\dots
i_{p-1}n}\tau_{i^{(j)}_{1}\dots
i^{(j)}_{\alpha_{j}}}=\tau_{i_{1}\dots i_{p-1}i^{(j)}_{2}\dots
i^{(j)}_{\alpha_{j}}i^{(j)}_{1}n}.
$$
Hence all the summands except
$C_{n}^{\alpha,p}$ are linear combinations of the elements
$C_{n}^{\beta,q}$ with
 $|\beta|+q=|\alpha|+p-1<|\alpha|+p$.
This immediately implies the equality
$SGZ(\frak Y)=\mathbb{C}[\pi_{1},\pi_{2},\dots,\pi_{n}]$.
Hence it suffices to show that
$SZ(\frak A)$ coincides with
$ \mathbb C [\pi_{1}^2,\pi_{2}^2,\dots,\pi_{n}^2]$. We have
$$
\pi_{i}^{2}=i-1-\sum_{k, l< i, k\ne l}\tau_{kli}=\sum_{1\le i,j,k\le i-1}\tau_{ijk}-\sum_{1\le i,j,k\le i}\tau_{ijk},
$$
where the indices
$i,j,k$ in the last sums are pairwise distinct. Therefore $\pi_{i}^{2}\in SZ(\frak Y)$.
On the other hand, we have already proved  that
  $SZ(A_{i})\subset \mathbb{C}[SZ(A_{i-1}),\pi_{i}^2]$.
\end{proof}

In a similar way we can prove the following theorem.

\begin{thm}\label{SC1} The supercentralizer
$SZ(\frak A_{n},\frak A_{n-2})$ is generated by the supercenter
$SZ(\frak A_{n-2})$ and the elements $\pi_{n},\pi_{n-1},\tau_{n-1}$.
\end{thm}

\smallskip\noindent
{\bf Remark}. According to Theorem~\ref{SC2} and the previous corollary,
every irreducible $\frak A_{n}$-module is the direct sum of pairwise nonisomorphic
irreducible $SGZ(\frak Y)$-modules, each being of the form
$V(a_{1},\dots,a_{n})$, where $a_{1},\dots,a_{n}$ are eigenvalues of the
elements $\pi_{1}^2,\pi_{2}^2,\dots,\pi_{n}^2$.
\smallskip

\begin{corollary} Consider the chain of $\Bbb Z_{2}$-graded algebras
$$\mathbb{C}\subset \mathcal{C}_{1}\otimes\frak A_{1}\subset
\mathcal{C}_{2}\otimes\frak A_{2}\subset\dots\subset\mathcal{C}_{n}\otimes\frak A_{n}.
$$
Then the algebra generated by the supercentralizers of this chain
coincides with the algebra
$\mathcal{C}_{n}\otimes\mathbb C [\pi_{1},\pi_{2},\dots,\pi_{n}]$,
and the algebra $SZ(\frak Y)$ generated by the supercenters
coincides with  $ \mathbb C
[\pi_{1}^2,\pi_{2}^2,\dots,\pi_{n}^2]$.
\end{corollary}

The following lemma describes some relations in the algebra
$SZ(\frak A_{n},\frak A_{n-2})$.
Consider the following elements of the algebra $\frak A_{n}$:
$$
F_{i}=\tau_{i}(\pi_{i}^2-\pi_{i+1}^2)+(\pi_{i+1}-\pi_{i}), \quad i=1,2,\dots, n-1.
$$

\begin{lemma} For $i=1,2,\dots,n-1$, the following relations hold:
$$
\tau_{i}\pi_{i}+\pi_{i+1}\tau_{i}=1,\quad(\pi_{i}-\pi_{i+1})\tau_{i}=\tau_{i}(\pi_{i}-\pi_{i+1}),
$$
$$
(\pi_{i}^2-\pi_{i+1}^2)\tau_{i}+\tau_{i}(\pi_{i}^2-\pi_{i+1}^2)=2(\pi_{i}-\pi_{i+1}),
$$
$$
F_{i}\pi_{i}+\pi_{i+1}F_{i}=0,\quad
F_{i}\pi_{i+1}+\pi_{i}F_{i}=0,\quad
F_{i}^2=\pi_{i}^2+\pi_{i+1}^2-(\pi_{i}^2-\pi_{i+1}^2)^2.
$$
\end{lemma}

\begin{proof}We have
$$
\tau_{i}\pi_{i}+\pi_{i+1}\tau_{i}=\tau_{ii+1}\sum_{j<i}\tau_{ji}+
\left(\tau_{ii+1}+\sum_{j<i}\tau_{ji+1}\right)\tau_{ii+1}
$$
$$
=1+\sum_{j<i}\tau_{i+1ij}-\sum_{j<i}\tau_{ji+1i}=1,
$$
and the first relation is proved. The remaining relations easily
follow from the first one.
\end{proof}

The following lemma is needed for obtaining analogs of Young's formulas;
it is a simple exercise in representation theory. Consider the algebra $H$
generated by a semisimple commutative algebra $A$ and elements
$\tau, p,q $ that commute with $A$ and satisfy the relations
$\tau^{2}=1$, $pq+qp=0$, $\tau p+q\tau=1$. It is easy to check that the
element $\Delta=p^{2}+q^2-(p^2-q^2)^2$
belongs to the center of   $H$.

\begin{lemma}\label{IR}
Let $V$ be an irreducible module over $H$ that is semisimple as an
$A[p,q]$-module. Then

$1)$ If $\Delta=0$ in  $V$, then $V$ is irreducible as an $A[p,q]$-module and
$p^2, q^2$ act in $V$
by multiplications by $a,b$, where
$a\ne b$, $a+b=(a-b)^2$, and $\tau$ acts as the operator $
\frac{p-q}{a-b}$.

$2)$ If $\Delta\ne0$ in $V$, then $V$ is the direct sum of two irreducible
$A[p,q]$-modules, in one of which (say, $U$) the elements
$p^2, q^2$ act by multiplications by $a,b$, where $a\ne
b,$ $a+b\ne(a-b)^2$, and
$$
V=H\otimes_{A[p,q]}U.
$$
\end{lemma}

\begin{corollary}\label{Young}
Let $V$ be an irreducible module over
$\mathcal C_{k}\otimes\frak A_{n}$ and  $ V(a_{1},\dots,a_{n})$
be the subspace of common eigenvectors for
 $\pi_{1}^2,\pi_{2}^2,\dots,\pi_{n}^2$ with eigenvalues
$a_{1},\dots,a_{n}$. Then

$1)$ $a_{i}\ne a_{i+1}$ for $i=1,\dots,n-1$.

$2)$ If $a_{i}+a_{i+1}=(a_{i}-a_{i+1})^2$, then
$ \tau_{i}$ acts in the subspace
$V(a_{1},\dots,a_{n})$ as the operator $\frac{\pi_{i}-\pi_{i+1}}{a_{i}-a_{i+1}}$.

 $3)$  If $a_{i}+a_{i+1}\ne(a_{i}-a_{i+1})^2 $, then
 $$ V(a_{1},\dots,a_{i+1},a_{i},\dots,a_{n})\ne0.$$
\end{corollary}

\begin{proof}
Consider $V$ as a module over the subalgebra
$SZ(\frak A_{i+1}, \frak A_{i-1})$, which is semisimple by Lemma~\ref{SC1}.
Consider the corresponding action of the algebra
$H(p,q,\tau)$, which is also semisimple. Decompose the module $V$ into
isotypic components with respect to $H(p,q,\tau)$; then the subspace
$V(a_{1},\dots,a_{n})$ is contained in the
isotypic component in which
the element $\Delta=p^{2}+q^2-(p^2-q^2)^2$ acts as $0$. Therefore,
by Lemma~\ref{IR}, $a_{i}\ne
a_{i+1}$ and $\tau_{i}$ acts in $V(a_{1},\dots,a_{n})$ as the operator
$\frac{\pi_{i}-\pi_{i+1}}{a_{i}-a_{i+1}}$. This proves $2)$. Let us prove
$3)$. Similarly to the above, we consider the
isotypic component with a fixed
value $\Delta\ne0$. By Lemma~\ref{IR}, $ V(a_{1},\dots,a_{i+1},a_{i},\dots,a_{n})\ne0 $.
\end{proof}

\def\SSpec{{\rm SSpec}}
\def\Spec{{\rm Spec}}

Denote by $\SSpec(n)$  the set of all possible sequences of eigenvalues of the
elements $\pi_{1}^2,\pi_{2}^2,\dots,\pi_{n}^2$, and by
$\Spec(n)$  the set of all possible sequences of
eigenvalues of the ordinary YJM-elements.
Also denote by $\Spec^+(n)$ the subset in
$\Spec(n)$ consisting of $(c_{1},c_{2},\dots,c_{n})$ such that
$c_{i}\ge0$ for $i=1,\dots,n$.

 \begin{thm} Let  $(a_{1},a_{2},\dots,a_{n})\in \SSpec(n)$.
Then

{\rm (i)}   $a_{i}\ge 0$ for every $i=1,\dots,n$, and there exists a unique
nonnegative integer
$b_{i}$ such that  $a_{i}=\frac{1}{2}b_{i}(b_{i}+1)$;

{\rm (ii)} the map
$$
(a_{1},a_{2},\dots,a_{n})\longrightarrow (b_{1},b_{2},\dots,b_{n})
$$
is a bijection of  $\SSpec(n)$ onto  $\Spec^+(n)$.
\end{thm}

\begin{proof} Let us prove the first assertion by induction on $i$. The
uniqueness is obvious. Let us prove the existence. If
$i=1$, then $\pi_{1}=0$ and $a_{1}=0$, whence $b_{1}=0$. If
$i=2$, then $a_{2}=\pi_{2}^2=1$ and $b_{2}=1$.
Assume that $i>1$ and
$a_{i}=\frac{1}{2}b_{i}(b_{i}+1)$ where
$b_{i}$ is a nonnegative integer. Then two cases are possible:

(a) $a_{i}+a_{i+1}= (a_{i}-a_{i+1})^2$;

(b) $a_{i}+a_{i+1}\ne (a_{i}-a_{i+1})^2$.

In case (a)  we have the following equation on $a_{i+1}$:
$$
a_{i+1}^2-(2a_{i}+1)a_{i+1}+a_{i}^2-a_{i}=0.
$$
By the induction hypothesis, this equation has two roots,
$ a_{i+1}=\frac{1}{2}(b_{i}+1)(b_{i}+2)$ and
$a_{i+1}=\frac{1}{2}(b_{i}-1)b_{i}$, and
 $b_{i}\ge0$. Therefore  $a_{i+1}\ge0$. In case (b),
 $(a_{1},\dots ,a_{i+1},a_{i},\dots,a_{n})\in \SSpec(n)$,
and the assertion follows from the induction hypothesis.

Now let us prove the second claim of the theorem. Recall the following characterization of the set
$\Spec(n)$. According to \cite{OV}, $(c_{1},c_{2},\dots,c_{n})\in \Spec(n)$
if and only if the following conditions hold:

$1)$ $c_{i}\in \Bbb Z$,\; $c_{1}=0$;

$2)$ $c_{i}\ne c_{i+1}$;

$3)$ for every $i=1,\dots,n-2$, we have
$(c_{i},c_{i+1},c_{i+2})\ne(d,d+1,d)$ for any
 $d\in\Bbb Z$;

$4)$ if $c_{i+1}\ne c_{i}\pm1$, then
$(c_{1},\dots,c_{i+1},c_{i}\dots,c_{n})\in \Spec(n)$.

Let us prove that  $b_{i}$, for $i=1,\dots,n$, satisfy the same conditions.

Assertions $1)$, $2)$, and $4)$ have already been proved.
Let us prove assertion $3)$. Assume that there exists  $i$ such that
$(b_{i},b_{i+1},b_{i+2})= (b,b+1,b)$. Then
$$
(a_{i}-a_{i+1})^2=(-1-b)^2=a_{i}+a_{i+1},\quad
(a_{i+1}-a_{i})^2=(1+b)^2=a_{i+1}+a_{i+2}.
$$
Hence, according to Corollary~\ref{Young},
$$
\tau_{i}=\frac{\pi_{i}-\pi_{i+1}}{a_{i}-a_{i+1}},\quad
\tau_{i+1}=\frac{\pi_{i+1}-\pi_{i+2}}{a_{i+1}-a_{i+2}}.
$$
But it is not difficult to check that these relations contradict the relation
 $$
\tau_{i}\tau_{i+1}\tau_{i}=\tau_{i+1}\tau_{i}\tau_{i+1},
$$
and assertion $4)$ is proved. Hence the map from assertion
(ii) of the theorem is an injection. The remaining part of the proof
follows the same scheme as for the symmetric group. First we prove that
there exists a sequence of admissible\footnote{A transposition of indices
is called admissible if it does not cause a violation of conditions
1)--4), see \cite{OV}.} transpositions
that sends a given
strict standard tableau to the standard tableau of the same shape
obtained by placing the integers $1,2,{\ldots} ,n$ successively into
the cells of the first row, then the cells of the second row, etc.
This implies that if the weight of a standard tableau has the same shape
as a
weight of some irreducible module, then it is itself a weight
of this module. Computing the number of irreducible representations and
the number of strict partitions completes the proof.
\end{proof}

Let us proceed to the description of analogs of Young's formulas. In contrast to
the ordinary Young's formulas, we do not use the Gelfand--Tsetlin basis,
but write the corresponding formulas in terms of the Gelfand--Tsetlin superalgebra.
Note that here we obtain a complete description of the action of the algebra
$\mathcal C_{k}\otimes\frak
A_{n}$ in irreducible modules.

Fix a strict partition $\alpha$ and choose a
subspace
$V^{\alpha}_{T_{0}}$ in the irreducible module
$V^{\alpha}$ corresponding to the standard tableau obtained
by filling the diagram row-wise. Let
${T}$ be an arbitrary tableau of shape $\alpha$. Then there exists a unique
$s\in S_{n}$ such that $T=sT_{0}$.
Denote by $P_{T}$ the map $V^{\alpha}_{T_{0}}\rightarrow V^{\alpha}_{T}$
that is the composition of $s$ and the projection of $V^{\alpha}$ to
$V^{\alpha}_{T}$ parallel to $\oplus_{T^{\prime}\ne
T}V^{\alpha}_{T^{\prime}}$.

\begin{thm}[Young's seminormal form for the algebra
$\mathcal C_{k}\otimes\frak A_{n}$]
Let $\alpha$ be a strict partition and ${T}$ be an arbitrary tableau of shape
$\alpha$. Then the commutation between
$\tau_{i}\in \frak A_{n}$ and $P_T$
is given by the following formulas:

$(i)$
If $a_{i}+a_{i+1}=(a_{i}-a_{i+1})^2$, then
 $$
 \tau_{i}P_{T}=\frac{\pi_{i}-\pi_{i+1}}{a_{i}-a_{i+1}}P_{T}.
 $$

 $(ii)$ If $a_{i}+a_{i+1}\ne(a_{i}-a_{i+1})^2 $ and $l(s_{i}T)> l(T)$, then
  $$
 \tau_{i}P_{T}=\frac{\pi_{i}-\pi_{i+1}}{a_{i}-a_{i+1}}P_{T}+\frac{1}{\sqrt2}\left(p_{i}-p_{i+1}\right)P_{s_{i}T}.
 $$

$(iii)$  If $a_{i}+a_{i+1}\ne(a_{i}-a_{i+1})^2 $ and $l(s_{i}T)<l(T)$, then
$$
 \tau_{i}P_{T}=\frac{\pi_{i}-\pi_{i+1}}{a_{i}-a_{i+1}}P_{T}+\frac{1}{\sqrt2}\left(p_{i}-p_{i+1}\right)\left(1-\frac{a_{i}+a_{i+1}}{(a_{i}+a_{i+1})^2}\right)P_{s_{i}T}.
$$
\end{thm}

\begin{proof}
First let us write an explicit formula for the map $P_{T}$. Let $T=sT_{0}$,
and let $s=s_{i_{1}}\dots s_{i_{l}}$ be the
reduced decomposition of a permutation $s$,
all transpositions in this decomposition being admissible. Then
\begin{equation}\label{explicit}
P_{sT_{0}}=P_{s_{i_{1}}}P_{s_{i_{2}}}\dots P_{s_{i_{l}}},
\end{equation}
where
$$
P_{s_{i}}=-\frac{p_{i}-p_{i+1}}{\sqrt 2}\left(\tau_{i}-\frac{\pi_{i}-\pi_{i+1}}{\pi^2_{i}-\pi^2_{i+1}}\right).
$$
It suffices to consider the case $l=1$ and check that if
$v\in V^{\alpha}_{T}$ and $s_{i}$ is an admissible transposition, then
$s_{i}v-v^{\prime}=P_{s_{i}}v\in
V^{\alpha}_{s_{i}T}$, where $v^{\prime}\in V^{\alpha}_{T}$.
But it is easy to check that
$\pi_{i}^2P_{s_{i}}=P_{s_{i}}\pi_{i+1}^2$,
$\pi_{i}^2P_{s_{i+1}}=P_{s_{i}}\pi_{i}^2$. Hence
$P_{s_{i}}v\in V^{\alpha}_{s_{i}T}$. Further, assertion  $(i)$
follows from Lemma~\ref{IR}.
Assertion $(ii)$ follows from the relation $P_{s_{i}T}=P_{s_{i}}P_{T}$
if $l(s_{i}T)>
l(T)$. Assertion  $(iii)$ follows from $(ii)$. The theorem is proved.
\end{proof}

\begin{remark}{\rm It is easy to check the relations
$P_{sT_{0}}\pi_{i}=\pi_{s(i)}P_{sT_{0}}$,
$P_{sT_{0}}p_{i}=p_{s(i)}P_{sT_{0}}$. Together with the
relations from the previous theorem, they give a complete description of the
action of the algebra $\mathcal
C_{k}\otimes\frak A_{n}$ in irreducible modules and coincide (after simple
transformations) with the formulas obtained in  \cite{N2} by other methods.}
\end{remark}

Let us proceed to the description of analogs of Young's formulas for the algebra
$\frak A_{n}$. For this we need to define analogs of the maps
$P_{T}$. In contrast to the previous case, we cannot do this using the
projection to the corresponding eigenspace. Instead, we will write
an analog of formula~(\ref{explicit}), which gives an explicit decomposition of this operator.

Fix a strict partition $\alpha$, choose subspaces
$V^{\alpha}_{T}$, $V^{\alpha}_{s_{i}T}$, where $s_{i}$
is an admissible transposition, and set
 $$
 Q_{s_{i}}=\left(\tau_{i}-\frac{\pi_{i}-\pi_{i+1}}{\pi^2_{i}-\pi^2_{i+1}}\right)\left(\frac{\pi_{i}}{\sqrt{\pi^2_{i}}}-\frac{\pi_{i+1}}{\sqrt{\pi^2_{i+1}}}\right)
$$if $a_{i}a_{i+1}\ne0$, and
$$
Q_{s_{i}}=\left(\tau_{i}-\frac{\pi_{i}-\pi_{i+1}}{\pi^2_{i}-\pi^2_{i+1}}\right)\frac{\pi_{i}+\pi_{i+1}}{\sqrt{\pi^2_{i}}+\sqrt{\pi^2_{i+1}}}
$$if $a_{i}a_{i+1}=0$.

As above, let $V^{\alpha}_{T_{0}}$ be the subspace in the irreducible module
$V^{\alpha}$ corresponding to the standard tableau filled by rows.
Let ${T}$ be an arbitrary tableau of shape
$\alpha$. Then there exists a unique $s\in S_{n}$ such that
$T=sT_{0}$ and $s=s_{i_{1}}\dots s_{i_{l}}$
is the reduced decomposition of the permutation $s$,
and all transpositions in this decomposition are admissible. Set
$$
Q_{T}=Q_{s_{1}}Q_{s_{2}}\dots Q_{s_{l}}.
$$

\begin{thm}[Young's seminormal form for the algebra $\frak A_{n}$] Let $\alpha$
be a strict partition and ${T}$ be an arbitrary tableau of shape $\alpha$.
Then $Q_{T}$ does not depend on the choice of a reduced decomposition of $s$
into the product of admissible transpositions, and the commutation
between $\tau_{i}\in\frak A_{n}$ and $Q_T$
is given by the following formulas:

$(i)$
 If $a_{i}+a_{i+1}=(a_{i}-a_{i+1})^2$, then
 $$
 \tau_{i}Q_{T}=\frac{\pi_{i}-\pi_{i+1}}{a_{i}-a_{i+1}}Q_{T}.
 $$

 $(ii)$ If $a_{i}+a_{i+1}\ne(a_{i}-a_{i+1})^2$, $a_{i}a_{i+1}\ne0$, and $l(s_{i}T)> l(T)$, then
  $$
 \tau_{i}Q_{T}=\frac{\pi_{i}-\pi_{i+1}}{a_{i}-a_{i+1}}Q_{T}-\frac{1}{2}\left(\frac{\pi_{i}}{\sqrt{a_{i+1}}}-\frac{\pi_{i+1}}{\sqrt{a_{i}}}\right)Q_{s_{i}T}.
 $$

 If  $l(s_{i}T)<l(T)$, then
$$
 \tau_{i}Q_{T}=\frac{\pi_{i}-\pi_{i+1}}{a_{i}-a_{i+1}}Q_{T}-\frac{1}{2}\left(\frac{\pi_{i}}{\sqrt{a_{i+1}}}-\frac{\pi_{i+1}}{\sqrt{a_{i}}}\right)\left(1-\frac{a_{i}+a_{i+1}}{(a_{i}+a_{i+1})^2}\right)Q_{s_{i}T}.
$$

$(iii)$ If $a_{i}+a_{i+1}\ne(a_{i}-a_{i+1})^2$, $a_{i}a_{i+1}=0$, and $l(s_{i}T)> l(T)$, then
$$
\tau_{i}Q_{T}=\frac{\pi_{i}-\pi_{i+1}}{a_{i}-a_{i+1}}Q_{T}-\frac{\pi_{i}+\pi_{i+1}}{\sqrt{a_{i}}+\sqrt{a_{i+1}}}Q_{s_{i}T}.
$$
If  $l(s_{i}T)<l(T)$, then
$$
\tau_{i}Q_{T}=\frac{\pi_{i}-\pi_{i+1}}{a_{i}-a_{i+1}}Q_{T}+\frac{\pi_{i}+\pi_{i+1}}{\sqrt{a_{i}}+\sqrt{a_{i+1}}}\left(1-\frac{a_{i}+a_{i+1}}{(a_{i}+a_{i+1})^2}\right)Q_{s_{i}T}.
$$
\end{thm}

\begin{proof} First let us check that $Q_{T}$ does not depend on the choice
of a reduced decomposition. For this it suffices to check the equality
$$
Q_{s_{i}}Q_{s_{i+1}}Q_{s_{i}}=Q_{s_{i+1}}Q_{s_{i}}Q_{s_{i+1}},
$$
which can be done by directly enumerating all possible cases. Then the proof
follows the scheme of the previous theorem.
\end{proof}

Partially supported by  EPSRC (grant EP/E004008/1).


\begin{thebibliography}{99}
\bibitem[C]{C}Cherednik I. V., Special bases of irreducible representations of a degenerate
affine Hecke algebra (Russian),  Funktsional. Anal. i Prilozhen.  20  (1986),  no. 1, 87--88.

\bibitem[FH]{FH}
Fulton W., Harris J., {\it Representation Theory.  A First Course}.
Graduate Texts in Mathematics 129, Springer-Verlag, New York, 1991.

\bibitem[J]{J}
Jozefiak T., Semisimple superalgebras, Lecture Notes Math. 1352,
1988, 96--113.

\bibitem[Ju1]{Ju1}
Jucys A., Symmetric polynomials and the center of the symmetric
group ring, Report Math. Phys. 5, 1974, 107--112.

\bibitem[Ju2]{Ju2}
Jucys A., Factorization of Young's projection operators for
symmetric groups, Litovsk. Fiz. Sb.  5, 1971, 1--10.

\bibitem[Mo]{Mo}
Morris A.,
Projective representations of finite reflection groups. III.  Comm. Algebra  32,  no. 7  (2004), 2679--2694.

\bibitem[M]{M}
 Murphy G., A new construction of the Young seminormal
representations of the symmetric group, J. Algebra 69 (1981),
287--291.

\bibitem[N1]{N1}
Nazarov M., Young's orthogonal form of irreducible projective representations
of the symmetric group, J. London Math. Soc. (2), 42 (1990), 437--451.

\bibitem[N2]{N2}
Nazarov M., Young's symmetrizers for projective representations of
the symmetric group, Adv. Math. 127, no. 2 (1997), 190--257.

\bibitem[NS]{NS}
Nazarov M., Sergeev A., Centralizer construction of the Yangian of the queer Lie superalgebra.
Studies in Lie theory,  417--441, Progr. Math. 243, Birkhauser Boston, Boston, MA, 2006.

\bibitem[OV]{OV}
Okounkov A., Vershik A., A new approach to representation theory of
symmetric groups. Selecta Math. (N. S.) 2, no. 4 (1996),
581--605.

\bibitem[Se]{Se}
Sergeev A., The Howe duality and the projective representations of
symmetric groups,
 Represent. Theory 3 (1999), 416--434.

\bibitem[Se1]{Se1}
Sergeev A., The invariant polynomials on simple Lie superalgebras,
Represent. Theory 3 (1999), 250--280.

\bibitem[Se3]{Se3}
Sergeev A.\ N., Tensor algebra of the identity representation as a module over the
Lie superalgebras ${\rm Gl}(n,\,m)$ and $Q(n)$ (Russian),  Mat. Sb. (N.S.)  123(165),  no. 3 (1984), 422--430.

\bibitem[KL]{KL} Kleshchev A., Linear and projective representations of symmetric groups,
Cambridge Tracts in Mathematics 163, Cambridge University Press, Cambridge, 2005. xiv+277 pp.

\bibitem[Sch]{Shc}
Schur J., On the representation of the symmetric and alternating groups by fractional
linear substitutions. Translated from the German
[J. Reine Angew. Math. 139 (1911), 155--250] by Marc-Felix Otto.
Internat. J. Theoret. Phys.  40 (2001),  no. 1, 413--458.

\bibitem[VO]{VO} Vershik A., Okounkov A., A new approach to representation theory of
symmetric groups. II, J. Math. Sci. (N.Y.) 131, no. 2 (2005), 5471--5494.

\bibitem[V]{V} Vershik A., A new approach to representation theory of
symmetric groups. III. Induced representation and Frobenius correspondence,
Moscow Math. J. 6, no. 3 (2006), 567--585.

\bibitem[V1]{V1} Vershik A., Local algebras and a new version of Young's orthogonal form,
In ``Topics in Algebra, part 2: Commutative Rings and Algebraic Groups'' (Warsaw 1988), Banach Cent.
Publ. 26, Part 2 (1990), 467--473.

\bibitem[VK]{VK}
Vershik A.\ M., Kerov S.\ V, Locally semisimple algebras. Combinatorial theory and the $K\sb
0$-functor (Russian), Current problems in mathematics. Newest results, Vol. 26,  3--56, 260, Itogi
Nauki i Tekhniki, Akad. Nauk SSSR, Vsesoyuz. Inst. Nauchn. i Tekhn. Inform., Moscow, 1985.

\bibitem[KW]{KW}
Khongsap T., Wang W., Hecke--Clifford algebras and spin Hecke algebras I,
{\tt arXiv:07040201}.

\bibitem[VV]{VV} Vershik A.\ M., Vsemirnov M.\ A., The local stationary presentation of the alternating
groups and normal form, {\tt arXiv:math/0703278} (submitted to J. Algebra).
\end{thebibliography}
\end{document}